\documentclass[10pt,letterpaper]{amsart}
\usepackage{amsmath}
\usepackage{amssymb}
\usepackage{amsfonts}
\usepackage{amsthm}
\usepackage{cancel}
\usepackage[all]{xy}
\usepackage[neveradjust]{paralist}
\usepackage{mdframed}
\usepackage{enumitem}
\usepackage{enumerate}
\usepackage[section]{placeins}
\usepackage{tikz}
\usepackage{import}
\usepackage{xifthen}
\usepackage{pdfpages}
\usepackage{transparent}
\usepackage{faktor}
\usepackage{xcolor}
\usepackage{calc}
\usepackage{graphicx,wrapfig}
\usepackage[
    colorlinks=false,
    pdfborder={0 0 0},
    linkcolor=Red,
]{hyperref}
\usepackage{caption}
\usepackage{pinlabel}
\usepackage{adjustbox}
\usepackage{microtype}
\usetikzlibrary{arrows,chains,matrix,positioning,scopes,cd}

\newcommand{\Z}{\mathbb{Z}}
\newcommand{\Q}{\mathbb{Q}}
\newcommand{\R}{\mathbb{R}}

\newcommand{\F}{\mathbb{F}}

\newcommand{\s}{\mathfrak{s}}

\theoremstyle{plain}
\newtheorem{theorem}{Theorem}[section]
\newtheorem{lemma}[theorem]{Lemma}
\newtheorem{corollary}[theorem]{Corollary}

\theoremstyle{definition}
\newtheorem{definition}[theorem]{Definition}

\newtheorem{proposition}[theorem]{Proposition}

\theoremstyle{remark}
\newtheorem*{remark}{Remark}

\begin{document}

\title{Integral Klein bottle surgeries and Heegaard Floer homology}

\author{Robert DeYeso III}

\begin{abstract}
In this paper, we study which closed, connected, orientable three-manifolds $X$ containing a Klein bottle arise as integral Dehn surgery along a knot in $S^3$. Such $X$ are presentable as a gluing of the twisted $I$-bundle over the Klein bottle to a knot manifold, and we use a variety of Heegaard Floer type invariants to generate surgery obstructions. Suppose that $X$ is $8$-surgery along a genus two knot, and arises by gluing the twisted $I$-bundle over the Klein bottle to an $S^3$ knot complement. We show that $X$ is an L-space, it must be the dihedral manifold $\left(-1; \tfrac{1}{2}, \tfrac{1}{2}, \tfrac{2}{5}\right)$, and the surgery knot must be $K=T(2,5)$.
\end{abstract}

\maketitle

\section{Introduction}

For a knot $K \subset S^3$, denote the result of $\frac{p}{q}-$Dehn surgery on $K$ as $S^3_{p/q}(K)$. A celebrated theorem of Lickorish and Wallace shows that any closed, orientable $3$-manifold may be constructed by integral Dehn surgery on a link in $S^3$ \cite{Lic62}\cite{Wal60}. It is then natural to ask which manifolds may be realized as Dehn surgery on a knot in $S^3$, and we will focus on those containing Klein bottles.

Much is known about surgeries $S^3_{p/q}(K)$ containing Klein bottles. Gordon and Luecke \cite{GL95} showed that the surgery slope $p/q$ is integral when $K$ is hyperbolic, and Teragaito \cite{Te01} extended this condition to $K$ non-cabled and showed that $p$ is divisible by four. In \cite{IT03}, Ichihara and Teragaito gave bounds for $|p|$ in terms of the knot genus $g(K)$ when $K$ is non-cabled, and shortly after showed the same bound holds when $K$ is cabled, albeit allowing rational slopes \cite{IT05}. Their combined results show that if $S^3_r(K)$ contains a Klein bottle with $K$ non-trivial, then $|r| \leq 4g(K)+4$ with equality only occurring for specific knots. If $K$ is hyperbolic, then $|r| \leq 4g(K)$.

The lens spaces containing a Klein bottle are $L(4n, 2n \pm 1)$ \cite{BW69}, and Teragaito proved that a genus one knot admitting a surgery containing a Klein bottle is either a trefoil or a Whitehead double \cite{Te01}. When $g(K)=2$, the largest surgeries containing a Klein bottle are $S^3_{\pm 12}(K)$, for $K = T(2, \pm 5)$ or $K = T(2, \pm 3) \# T(2, \pm 3)$ due to \cite[Theorem 1]{IT05}. The next largest slope to consider is then $|r|=8$.

Suppose $X$ is realizable as $8$-surgery on a genus two knot $K$, and contains a Klein bottle. Observe that $X$ contains a regular neighborhood of a Klein bottle, that we will denote by $N$, which is the twisted $I$-bundle over the Klein bottle. While $X$ is can then be decomposed as a gluing of $N$ to a knot manifold, or rational homology solid torus, we will focus on gluings with $S^3$ knot complements so that $X = (S^3 \setminus \nu J) \cup_h N$ for some knot $J \subset S^3$. The gluing $h$ and its effects on $X$ are studied in Subsection \ref{sub:pairings}. By utilizing techniques involving knot Floer, Heegaard Floer, and bordered Heegaard Floer homology, we prove
\begin{theorem}
Let $X=S^3_8(K)$ with $g(K)=2$ contain a Klein bottle, and write $X = M \cup_h N$. If $M = S^3 \setminus \nu J$, then $X$ is the Seifert fibered manifold $(-1; \frac{1}{2}, \frac{1}{2}, \frac{2}{5})$ with base orbifold $S^2$, $J$ is the unknot, and $K=T(2,5)$.
\label{thm:main}
\end{theorem}

The theorem is stated in terms of positive surgery, but the analogous result for negative surgery holds with $K = T(2,-5)$. In order to obstruct the complements $S^3 \setminus \nu J$ in gluing, we use the large surgery theorem of Osv\'ath-Szab\'o and Rasmussen \cite{OS04b, Ras03} for Heegaard Floer homology; bordered Heegaard Floer invariants due to Lipshitz, Osv\'ath, and Thurston \cite{LOT18b}; and their immersed curves formulation developed by Hanselman, Rasmussen, and Watson \cite{HRW16, HRW18}. The versatility of the immersed curves package lends itself toward studying Dehn surgery problems, and has already led to fruitful results towards the cosmetic surgery conjecture (see \cite{Han19}).

The results in Theorem \ref{thm:main} actually hold for complements of knots $J$ in integer homology sphere L-spaces $Y$.  These are integer homology spheres $Y$ with the simplest Heegaard Floer homology, which is to say dim $\widehat{\textit{HF}}(Y) = |H_1(Y, \Z)|$. We briefly discuss this generalization in Section \ref{sec:YnotS3}. Figure \ref{fig:plus2RHTpair} shows the immersed curves machinery for $(S^3 \setminus \nu T(2,3)) \cup_h N$, where the count of intersection points corresponds to dim $\widehat{\textit{HF}}(X)$. This manifold (along with an integral family depending on $h$) are toroidal L-spaces, and were known to Hanselman, Rasmussen, and Watson in \cite{HRW18}. The main result of this paper shows that these toroidal manifolds cannot be realized as surgery along a knot.

\begin{figure}[!h]
\centering
\def\svgscale{1.1}
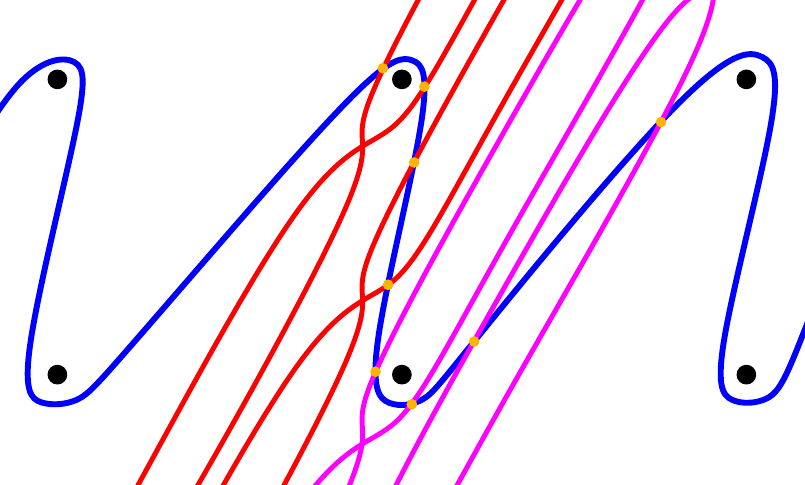
\caption{The pairing of immersed curves for $S^3 \setminus \nu T(2,3)$ in blue and $N$ in red and purple, that computes $\widehat{\textit{HF}}((S^3 \setminus \nu T(2,3)) \cup_h N)$.}
\label{fig:plus2RHTpair}
\end{figure}

Ichihara and Teragaito remark that the 2-bridge knot $6_2$ admits an $8$-surgery containing a Klein bottle \cite[Example 5.6]{IT03}. This knot is not an L-space knot, and so Theorem \ref{thm:main} then implies that $S^3_8(6_2)$ is obtained as $Y \setminus \nu J$ glued to $N$, with $Y$ not an integer homology sphere L-space. This example highlights that gluing along integer homology sphere L-space complements is special. Additionally, there are many examples of genus two cabled knots in $S^3$ admitting rational surgeries $S^3_{8/q}(K)$ that contain Klein bottles. When $K$ is a torus knot, the only example arises from the dihedral manifold obtained by Dehn filling $N$ stated in Theorem \ref{thm:main}. Otherwise $K$ is the $(2,1)$-cable of a genus one knot, and there are many rationally sloped fillings of the associated cable space for $S^3 \setminus \nu K$ that yield Klein bottles \cite[Corollary 7.3]{Gor83}.

\section*{Organization}

\,\,\,\,\,\,\, Throughout this paper we only consider positive surgeries, and remark that the analogous result for negative surgery follows by orientation reversal of both bordered invariants in pairing. Additionally, all manifolds are assumed to be compact, connected, oriented 3-manifolds, unless stated otherwise, and the coefficients in Floer homology are assumed to belong to $\F = \F_2$. We denote closed manifolds by $X$ or $Y$, and manifolds with (typically torus) boundary by $M$. Also knots $J \subset Y$ bounding a disk are said to be trivial, and figures will have the immersed curves invariant for knot complements in blue and the filling manifold in red.

Section \ref{sec:background} summarizes the relevant background from knot Floer, Heegaard Floer, and bordered Heegaard Floer homology. Additionally, it contains an overview of the immersed curves formulation of bordered Floer invariants (for manifolds with torus boundary), along with properties and symmetries.

Section \ref{sec:proof} introduces and constrains the possible gluings $h$, and establishes two of the three lemmas on the way to proving Theorem \ref{thm:main}. We handle the case that $J$ is trivial in Lemma \ref{lem:MLemFinite}, showing that $X$ must be a dihedral manifold that falls under Doig's classification of finite, non-cyclic surgeries for $p \leq 9$ in \cite{Doi15}. We then consider gluings with $J$ non-trivial in Lemma \ref{lem:MLemS3}, showing that $J$ must be a trefoil in order to possibly have the right Floer homology for $X$.

Section \ref{sec:gr} gives a brief overview of the refined grading on bordered invariants and their relation to the relative $\Q$-grading on $\widehat{\textit{HF}}$ under pairing. This extra structure is sufficient to establish the final lemma needed to complete the proof of Theorem \ref{thm:main}, showing that the toroidal gluings involving trefoil complements are obstructed. 

Section \ref{sec:YnotS3} mentions some generalizations when $Y$ is an integer homology sphere L-space, and then provides some homological constraints for general $Y$. Specifically, $Y$ must be an integer homology sphere.

\section*{Acknowledgements}
I thank Jonathan Hanselman, Jen Hom, Adam Levine, Robert Lipshitz, Liam Watson, and Biji Wong for their time in very helpful discussions. I would also like to thank Tye Lidman for suggesting a bordered Heegaard Floer approach to studying Klein bottle surgeries, and for giving incomparable guidance as an advisor. I was partially supported by NSF grant DMS-1709702.

\section{Background Material}
\label{sec:background}

We will assume the reader is familiar with the constructions of knot Floer homology \cite{OS04b, Ras03} and Heegaard Floer homology \cite {OS04d}, and just include relevant theorems before introducing bordered Heegaard Floer homology for manifolds with torus boundary (see \cite{LOT18b}) and its immersed curves formulation (see \cite{HRW16, HRW18}).

\subsection{Knot Floer and Heegaard Floer homology}
\label{subsec:2.1}

The Heegaard Floer homology theory of Ozsv\'ath and Szab\'o, and knot Floer homology theory of Ozsv\'ath and Szab\'o, and independently Rasmussen, have proven to be very powerful invariants of $3$-manifolds and knots. To an oriented knot $K$ in an integer homology sphere $Y$, Ozsv\'ath, Szab\'o, and Rasmussen associate a bi-graded, finitely generated vector space that decomposes as
\[
\displaystyle \widehat{\textit{HFK}}(Y, K) = \bigoplus_{m,a \in \Z} \widehat{\textit{HFK}}_m(Y, K, a).
\]
The integers $m$ and $a$ denote the Maslov (or homological) and Alexander gradings, respectively. We may often suppress the Maslov grading from notation when not needed, and the ambient 3-manifold from notation when $Y=S^3$.

For knots $K \subset S^3$, knot Floer homology categorifies the Alexander polynomial, and leads to knot genus and fiberedness detection results. Ozsv\'ath and Szab\'o showed in \cite{OS04b} that knot Floer homology detects the knot genus exactly via 
\[
g(K) = \max\{a \geq 0 \,\, | \,\, \widehat{\textit{HFK}}(K, a) \neq 0\}.
\]
Along with work of Ghiggini and Ni, it also detects precisely if a knot is fibered \cite{OS04e, Ghi08, Ni07}: 
\[
K \subset S^3 \,\, \text{fibered} \,\, \Leftrightarrow \widehat{\textit{HFK}}(K, g(K)) \cong \F.
\]
Further, knot Floer homology determines whether a knot in $S^3$ is a trefoil or figure-eight knot following from genus and fiberedness detection \cite{Ghi08}.

Given a closed, oriented 3-manifold $Y$, Ozsv\'ath and Szab\'o associate to it a finitely-generated vector space that decomposes as
\[
\widehat{\textit{HF}}(Y) = \bigoplus_{\s \in \text{Spin}^{c}(Y)} \widehat{\textit{HF}}(Y, \s),
\]
over $\text{spin}^{c}$ structures \cite{OS04d}. If $Y$ is a rational homology sphere, then dim $\widehat{\textit{HF}}(Y, \s) \geq 1$ for each $\s$. When we have equality for all $\s$, we say $Y$ is a (Heegaard Floer) L-space, generalizing the behavior exhibited by lens spaces. The Heegaard Floer groups admit a relative $\Z$-grading, also called their Maslov grading, that can be lifted to an absolute $\Q$-grading for a rational homology sphere together with the normalization that the generator of $\widehat{\textit{HF}}(S^3)$ has grading zero \cite{OS06}. For an L-space $Y$, the grading of the generator of $\widehat{\textit{HF}}(Y, \s)$ is an invariant of $Y$ known as its Heegaard Floer correction term, also known as its $d$-invariant, which we denote by $d(Y, \s)$ \cite{OS03a}.

To better understand $\widehat{\textit{HF}}(S^3_p(K), \s)$, we will use the so-called Large Surgery theorem. For a null-homologous knot $K$ in a 3-manifold $Y$, the full knot Floer complex $CFK^{\infty}(Y,K)$ is a $\Z \oplus \Z$-filtered chain complex. For our purposes, we may take $Y=S^3$ and simply write $CFK^{\infty}(K)$ in this setting. For $X$ a subset of $\Z \oplus \Z$, let $CX$ be the subgroup of $CFK^{\infty}(K)$ generated by those elements with filtration level $(i,j) \in X$. Further, for $s \in \Z$ define the subcomplexes 
\[
\widehat{\mathcal{A}}_s = C\{\max\{i,j-s\}=0\},
\]
and the respective homology groups $\widehat{A}_s = H_{\ast}(\widehat{\mathcal{A}}_s)$. This notation is suggestive of a connection to the hat-flavor of Heegaard Floer homology.

Let us identify $\text{Spin}^{c}(S^3_p(K))$ with $\Z/p\Z$ as in \cite[Subsection 2.4]{OS08}, and denote the correspondence using $[s] \in \text{Spin}^{c}(S^3_p(K))$ for $[s] \in \Z/p\Z$. The following theorem shows that $\widehat{\textit{HF}}(S^3_p(K),[s])$ and $\widehat{A}_s$ are isomorphic if $p$ is large relative to $s$.

\begin{theorem}[\cite{OS04b, Ras03}] For $p \gg 0$ and any $s \in \Z$ with $|s| \leq p/2$, there is an isomorphism
\begin{center}
$\widehat{\textit{HF}}(S^3_p(K)),[s]) \cong \widehat{A}_s$.
\end{center}
Here, $[s] \in \Z/p\Z$ is the corresponding $\text{spin}^{c}$ structure on $S^3_p(K)$.
\label{thm:largesurgery}
\end{theorem}

In particular, we may apply this theorem for all $\text{spin}^{c}$ structures of $S^3_8(K)$ with $g(K)=2$.

\begin{proposition}
Let $K \subset S^3$ have $g(K) = 2$. Then dim $\widehat{\textit{HF}}(S^3_8(K), [s]) = 1$ for at least five of the eight $[s] \in \text{Spin}^{c}(S^3_8(K))$.
\label{prop:HF5of8}
\end{proposition}

\begin{proof}
We have that $\widehat{A}_{-s}$ and $\widehat{A}_s$ are isomorphic due to Lemma 2.3 of [HLW15], following from the fact that $CFK^{\infty}(K)$ is filtered chain homotopy equivalent under reversing the roles of $i$ and $j$. Since $\widehat{\textit{HFK}}$ detects the knot genus, we have $\widehat{A}_s = \widehat{\textit{HF}}(S^3) = \F$ for $|s| \geq g(K)$. Since $S^3_8(K)$ is large surgery, Theorem \ref{thm:largesurgery} implies $\widehat{\textit{HF}}(S^3_8(K),[s]) \cong \F$ for $s \in \Z$ satisfying $[s] \neq 0, \pm 1 \in \Z/8\Z$. 
\end{proof}

This simple structure of $\widehat{\textit{HF}}(S^3_8(K))$ will be very useful towards establishing Theorem \ref{thm:main}, where we appeal to counting the number of $\mathfrak{t} \in \text{Spin}^{c}(X)$ supporting dim $\widehat{\textit{HF}}(X, \mathfrak{t}) > 1$ as we vary $J$. This large surgery is also particularly special because of the following proposition.

\begin{proposition}
$S^3_8(K)$ is irreducible for any knot $K$ with $g(K) = 2$.
\label{prop:reducible8}
\end{proposition}

\begin{proof}
To generate a contradiction, suppose that $S^3_8(K)$ is reducible. From \cite{MS03}, we see that $S^3_p(K)$ reducible implies $1 < |p| \leq 2g(K)-1$ for $K$ non-cabled. So it must be the case that $K$ is the $(r,s)$-cable of some knot $K'$, where $r$ and $s$ are  coprime and positive with $s>1$. The cabling conjecture holds for cable knots, and so the slope $p=rs$ provided by the cabling annulus is the only reducing slope for $S^3_p(K)$. In this case we have $S^3_{rs}(K) \cong L(s,r) \# S^3_{\frac{r}{s}}(K')$, and so
\[
S^3_{8}(K) \cong L(8,1) \# S^3_{\frac{1}{8}}(K').
\]

Let $[\s_i] \in \text{Spin}^{c}(S^3_8(K))$ restrict to $[\s'_j] \times [\s_0]$, where $[\s'_j] \in \text{Spin}^{c}(L(8,1))$ and $[\s_0] \in \text{Spin}^{c}(S^3_{\frac{1}{8}}(K'))$. The K\"unneth formula for the hat-flavor of Heegaard Floer homology \cite[Theorem 1.5]{OS04c} implies
\begin{align*}
\widehat{\textit{HF}}(S^3_8(K), [\s_i]) &= H_{\ast}(\widehat{CF}(L(8,1),[\s'_j]) \otimes_{\F} \widehat{CF}(S^3_{\frac{1}{8}}(K'), [\s])) \\
&= \widehat{\textit{HF}}(S^3_{\frac{1}{8}}(K'), [\s]),
\end{align*}
since $L(8,1)$ is a lens space. Theorem \ref{thm:largesurgery} forces dim $\widehat{\textit{HF}}(S^3_8(K), [\s_i]) = \text{dim} \, \widehat{A}_i$, and since $\text{dim} \, \widehat{A}_s = 1$ for $|s| \geq g(K)$, we see that $S^3_8(K)$ is an $L$-space and $K'$ is an $L$-space knot. From \cite[Proposition 9.5]{OS11}, the $\nu$ invariant for $K'$ must be trivial, which implies $K'$ is the unknot by \cite[Proposition 9.6]{OS11}.
Therefore $K$ is trivial as the $(1,8)$-cable of the unknot, and so $S^3_8(K) \cong L(8,1)$, yielding a contradiction.
\end{proof}

We also have the following immediate corollary, which is useful in Section \ref{sec:YnotS3} for the case when $Y \neq S^3$.

\begin{corollary}
Let $M = Y \setminus \nu J$ be a knot manifold. If $X = M \cup_h N$ is realizable as $S^3_8(K)$ for $g(K) = 2$, then $M$ is irreducible. 
\label{cor:irreduciblecomplement}
\end{corollary}

\subsection{Bordered Heegaard Floer invariants and the pairing theorem}
\label{subsec:BHF}

Bordered Heegaard Floer homology, introduced by Lipshitz, Ozsv\'ath, and Thurston, provides a way of computing $\widehat{\textit{HF}}(X)$ by decomposing $X$ along an essential surface, and then recovering its Floer homology by a suitable means of pairing the relative Floer invariants for the decomposed pieces \cite{LOT18b}. While defined for general manifolds with connected boundary, we will only be interested in applying the theory to manifolds with torus boundary.

Let $M$ be an orientable 3-manifold with torus boundary and choose $\alpha, \beta$ in $\partial M$ with $\beta \cdot \alpha = 1$, so that $(\alpha, \beta)$ forms a parameterization of $\partial M$. A bordered 3-manifold is such a triple $(M, \alpha, \beta)$. For a bordered 3-manifold $M_2$, they associate a differential module called a type D structure $\widehat{\textit{CFD}}(M_2, \alpha_2, \beta_2)$. For another bordered 3-manifold $M_1$, they associate a suitably dual object $\widehat{CFA}(M_1, \alpha_1, \beta_1)$, and prove Theorem \ref{thm:boxpairing} below showing that $\widehat{\textit{CF}}(M_1 \cup_h M_2)$ is modeled by the box tensor product of these modules (with corresponding parameterizations). 

\begin{theorem}[{\cite[Theorem 10.42]{LOT18b}}] Consider the pairing $X = M_1 \bigcup_h M_2$, where the $M_i$ are compact, oriented 3-manifolds with torus boundary and $h: \partial M_2 \rightarrow \partial M_1$ is an orientation reversing homeomorphism. Then
\[
\widehat{\textit{HF}}(X) \cong H_{\ast}(\widehat{CFA}(M_1, \alpha_1, \beta_1) \boxtimes \widehat{\textit{CFD}}(M_2, h^{-1}(\beta_1), h^{-1}(\alpha_1))),
\]
where the isomorphism is one of relatively graded vector spaces that respects the $\text{Spin}^{c}$ decomposition. 
\label{thm:boxpairing}
\end{theorem}

\begin{wrapfigure}{r}{0.4\linewidth}
\vspace{1\intextsep}
\labellist
\small\hair 2pt
\pinlabel $\iota_0$ at -3 25
\pinlabel $\iota_1$ at 128 24
\pinlabel $\rho_1$ at 60 56
\pinlabel $\rho_2$ at 60 32
\pinlabel $\rho_3$ at 60 8
\endlabellist
\centering
\includegraphics[scale=1]{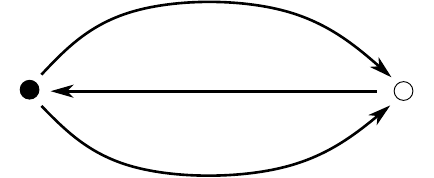}
\vspace{1\intextsep}
\caption{A quiver for $\mathcal{A}(\mathbb{T})$}
\label{fig:Aquiver}
\end{wrapfigure}
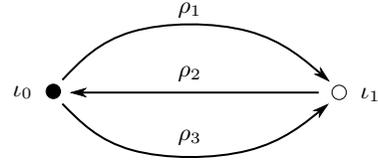

The type A and type D structures are types of modules over the differential graded torus algebra $\mathcal{A}(\mathbb{T})$. This algebra is generated over $\F$ by the elements $\rho_1, \rho_2, \rho_3$ and idempotents $\iota_0$ and $\iota_1$, and multiplication in $\mathcal{A}$ is described by the quiver in Figure \ref{fig:Aquiver} and further satisfies the relations $\rho_2\rho_1 = \rho_3\rho_2 = 0$. We will concatenate the multiplication using the common shorthand notation of $\rho_{12}=\rho_1\rho_2, \,\,\, \rho_{23}=\rho_2\rho_3$, and $\rho_{123}=\rho_1\rho_2\rho_3$. In this way, $\{\iota_0, \iota_1, \rho_1, \rho_2, \rho_3, \rho_{12}, \rho_{23}, \rho_{123}\}$ is an $\F$-basis for $\mathcal{A}$. Finally, let $\mathcal{I}$ denote the subring of idempotents.

The module $\widehat{\textit{CFA}}(M)$ is a type A structure, which is a right-$\mathcal{A}_{\infty}$ module over $\mathcal{A}$. The module structure comes with a family of maps
\[
m_{k+1}: \widehat{\textit{CFA}}(M) \otimes_{\mathcal{I}} \mathcal{A} \otimes_{\mathcal{I}} \cdots \otimes_{\mathcal{I}} \mathcal{A} \rightarrow \widehat{\textit{CFA}}(M),
\]
and we will use commas to separate the tensor factors of $m_{k+1}$ in the future. A type D structure over $\mathcal{A}$ is a left $\mathcal{I}$-module $V$ with a splitting over a left action of the idempotents $V \cong \iota_0 V \oplus \iota_1 V$, equipped with an $\mathcal{I}$-linear map $\delta^1: V \rightarrow \mathcal{A} \otimes V$. The map $\delta^1$ satisfies a compatibility condition that ensures $\partial(a \otimes \mathbf{x}) = a \cdot \delta^1 \mathbf{x}$ is a differential on $\mathcal{A} \otimes_{\mathcal{I}} V$. This gives $\mathcal{A} \otimes_{\mathcal{I}} V$ a left differential module structure over $\mathcal{A}$, and we will typically view $\widehat{\textit{CFD}}(M)$ as this differential module. This type D structure also comes with a collection of recursively defined maps $\delta^k: V \rightarrow \mathcal{A}^{\otimes k} \otimes V$ with $\delta^0: V \rightarrow V$ the identity and $\delta^{k} = (\text{id}_{\mathcal{A}^{\otimes (k-1)}} \otimes \partial \delta^1) \circ \delta^{(k-1)}$.

The chain complex $\widehat{CFA}(M_1, \alpha_1, \beta_1) \boxtimes \widehat{\textit{CFD}}(M_2, h^{-1}(\beta_1), h^{-1}(\alpha_1))$ in Theorem \ref{thm:boxpairing} is obtained from $\widehat{CFA}(M_1, \alpha_1, \beta_1) \otimes_{\mathcal{I}} \widehat{\textit{CFD}}(M_2, h^{-1}(\beta_1), h^{-1}(\alpha_1))$, and has differential $\partial^{\boxtimes}$ given by
\[
\partial^{\boxtimes} (\mathbf{x} \otimes \mathbf{y}) = \sum^{\infty}_{k=0}(m_{k+1} \otimes \text{id})(\mathbf{x} \otimes \partial^{k}(\mathbf{y}).
\]
This sum is finite if the type D structure is bounded, which is to say $\delta^{k}$ vanishes for sufficiently large $k$.
We will refer to it as the box tensor product of the type A and D structures involved. This is a computable model of the $\mathcal{A}_{\infty}$ tensor product and enables a proof of Theorem \ref{thm:boxpairing}. However, computations with the box tensor product are cumbersome, and so we will use a geometric interpretation of these invariants as often as possible.

\subsection{Bordered invariants as train tracks}
\label{subsec:IC}

In \cite{HRW16, HRW18}, Hanselman, Rasmussen, and Watson give a geometric construction of $\widehat{\textit{CFD}}(M)$ as a collection of immersed curves, and prove an analogue of the pairing theorem that uses these objects. This invariant lives in the punctured torus, which we now define.

\begin{definition}
Let the punctured torus $T_M$ be defined as $\left( H_1(\partial M; \R) / H_1(\partial M; \Z) \right) \setminus \{z\}$, where $z = (1-\epsilon, 1-\epsilon)$ for $\epsilon$ small. We refer to $z$ as the marked point, and orient $T_M$ so that the $y$-axis projects to $\alpha$ and the $x$-axis projects to $\beta$, with $\alpha, \beta$ specifying the handle decomposition of $\partial M \setminus z$.
\end{definition}

The prototype immersed curves invariant is a type of train track in $T_M$, a construction we briefly introduce. Along the way, we will apply these concepts to $\widehat{\textit{CFD}}(N)$, for $N$ the twisted $I$-bundle over the Klein bottle. Given a type D structure $\widehat{\textit{CFD}}(M, \alpha, \beta)$ we can conveniently express it as a decorated graph, with vertex set generating $V$ and edge set describing $\delta^1$. Length $k$ directed paths are associated to $\delta^k$, and correspondingly the type D structure is bounded if it admits no directed cycles. The vertices labeled $\bullet$ correspond to $\iota_0$ generators, and those labeled $\circ$ correspond to $\iota_1$ generators. For the edges, a term $\rho_{I} \otimes \mathbf{x}_2 \in \delta^1 \mathbf{x}_1$ gives a directed edge labeled $\{I\}$ from $\mathbf{x}_1$ to $\mathbf{x}_2$.

\begin{wrapfigure}{r}{0.4\linewidth}
\vspace{1\intextsep}
\labellist
\small\hair 2pt
\pinlabel $3$ at 31 68
\pinlabel $1$ at -3 32
\pinlabel $2$ at 31 -6
\pinlabel $123$ at 70 32
\pinlabel $12$ at 93 32
\pinlabel $12$ at 139 32
\endlabellist
\centering
\includegraphics[scale=1]{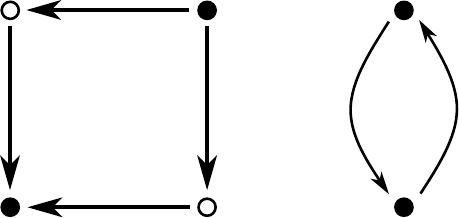}
\vspace{1\intextsep}
\caption{Left: $\widehat{\textit{CFD}}(N, \phi_1, \phi_0, \s_1)$. Right: $\widehat{\textit{CFD}}(N, \phi_1, \phi_0, \s_0)$.}
\label{fig:N0and1DstrNoTwist}
\end{wrapfigure}

The bordered invariant for $N$ is computed in \cite{BGW13} from a bordered Heegaard diagram (we caution the reader that the opposite idempotent decomposition of the torus algebra $\mathcal{A}(\mathbb{T})$ is used in this reference). Supporting two different Seifert structures, $\partial N$ is parameterized by the dual slopes $\phi_0$ and $\phi_1$ that correspond to the fiber slope of the structure with base orbifold a M\"obius band and $D^2(2,2)$, respectively. The slope $\phi_0$ is the rational longitude of $N$, or the unique slope in $\partial N$ that includes in $H_1(N)$ with finite order \cite{Wat12}. The curve $\phi_1$ includes in $H_1(N)$ as twice a primitive curve, and for this reason $N$ admits two torsion $\text{spin}^{c}$ structures $\s_0$ and $\s_1$. We will always assume this (standard) parameterization is taken before acting on $\partial N$. The decorated graph for $\widehat{\textit{CFD}}(N, \phi_1, \phi_0)$ is shown in Figure \ref{fig:N0and1DstrNoTwist}.

\begin{figure}[ht!]
\labellist
\small\hair 2pt
\pinlabel $\rho_1$ at 35 42
\pinlabel $z$ at 71 65
\pinlabel $\rho_2$ at 160 42
\pinlabel $z$ at 180 65
\pinlabel $\rho_3$ at 268 42
\pinlabel $z$ at 290 65
\pinlabel $\rho_{12}$ at 375 25
\pinlabel $z$ at 405 65
\pinlabel $\rho_{23}$ at 470 40
\pinlabel $z$ at 517 65
\pinlabel $\rho_{123}$ at 583 42
\pinlabel $z$ at 629 65
\endlabellist
\centering
\includegraphics[scale=0.65]{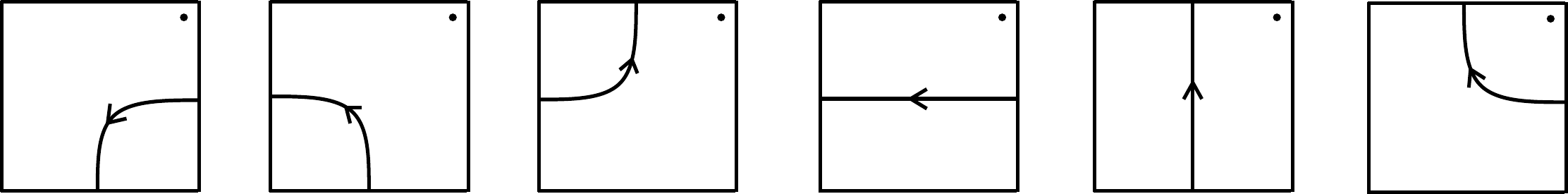}
\caption{Train track segments in $T_M$ corresponding to the $\rho_I \otimes$ terms appearing in $\delta^1$.}
\label{fig:ReebChords}
\end{figure}
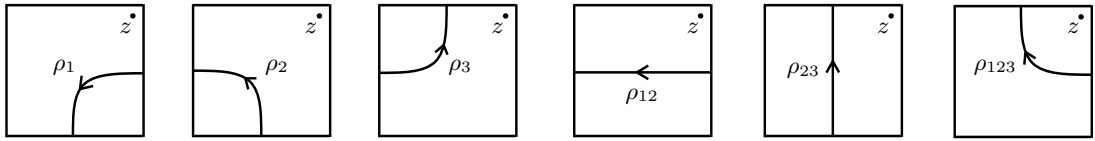

Given a decorated graph for $\widehat{\textit{CFD}}(M, \alpha, \beta)$, we may construct a train track $A(\theta_M)$ in $T_M$ as follows. First, embed the $\bullet$ vertices corresponding to $V_0 = \iota_0 V$ generators along $\alpha$ in the interval $0 \times [\frac{1}{4}, \frac{3}{4}]$, and the $\circ$ vertices corresponding to $V_1 = \iota_1 V$ generators along $\beta$ in the interval $[\frac{1}{4}, \frac{3}{4}] \times 0$. Next, the edges describing $\delta^1$ are embedded in $T_M$ according to Figure \ref{fig:ReebChords}. We refer to such an $A(\theta_M)$ as a type A realization of $\widehat{\textit{CFD}}(M, \alpha, \beta)$. For example, this process for $\widehat{\textit{CFD}}(N)$ is shown in Figure \ref{fig:N1and2Abase}. For pairing, we will also need the dual type D realization $D(\theta_M)$, which is generated by reflecting $A(\theta_M)$ across the anti-diagonal in $T_M$.

\begin{figure}[!ht]
\centering
\def\svgscale{0.6}
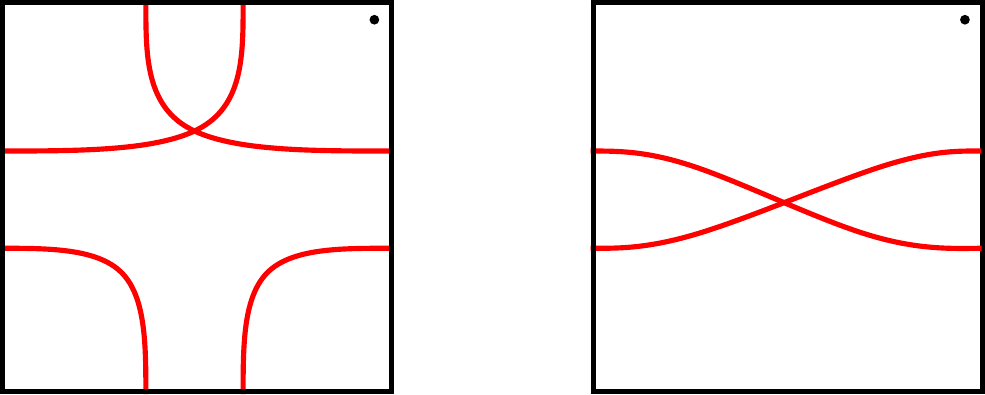
\caption{The type A realizations for both components of $\widehat{\textit{CFD}}(N, \phi_1, \phi_0)$.}
\label{fig:N1and2Abase}
\end{figure}

Now we describe the prototype geometric analogue of Theorem \ref{thm:boxpairing} for a gluing $M_1 \cup_h M_2$ from \cite{HRW16}. Divide $T_{M_1}$ into four quadrants, and include $A(\theta_{M_1})$ in the first quadrant and extend horizontally and vertically. Likewise, include $D(\theta_{M_2})$ into the third quadrant and extend horizontally and vertically. This is shown on the left side of Figure \ref{fig:T23N0BoxLift}. Let $\mathcal{C}(A(\theta_{M_1}), D(\theta_{M_2}))$ be the vector space over $\F$ generated by intersections between $A(\theta_{M_1})$ and $D(\theta_{M_2})$, and $d^{\theta}$ a linear map counting bigons analogous to the Whitney disks of Heegaard Floer homology. Hanselman, Rasmussen, and Watson show that under mild hypotheses, $\mathcal{C}(A(\theta_{M_1}), D(\theta_{M_2}), d^{\theta})$ forms a chain complex that may be identified with $\mathcal{C}(\widehat{CFA}(M_1, \alpha_1, \beta_1) \boxtimes \widehat{\textit{CFD}}(M_2, h^{-1}(\beta_1), h^{-1}(\alpha_1)), \partial^{\boxtimes})$ \cite[Theorem 16]{HRW16}. This prototype pairing theorem will be used in Section \ref{sec:gr} to handle grading information used to rule out the trefoil complement for gluing.

\subsection{$\widehat{\textit{HF}}\bf{(\textit{M})}$ and the pairing theorem}

In general, intersecting train tracks as previously described is not an invariant. This happens for instance if the decorated graph for $\widehat{\textit{CFD}}(M_i, \alpha_i, \beta_i)$ is not valence 2, so that extended type D structures and local systems need to be employed. This process culminates in an immersed curves invariant $\widehat{\textit{HF}}(M)$ (possibly decorated by local systems) that satisfies a full analogue of Theorem \ref{thm:boxpairing}, stated below in Theorem \ref{thm:ICpairing}. 

However in the examples we encounter, $\widehat{\textit{HF}}(M)$ is given by $A(\theta_M)$ and requires no extra decoration. This is the case when $M$ is \textit{loop type}, a class of manifold first introduced by Hanselman and Watson in \cite{HW15}. The twisted $I$-bundle over the Klein bottle $N$ is loop type, and the steps to prove the main theorem will ultimately force potential knot exteriors for gluing to be loop type as well. For this reason, we will not need to involve local systems.

\begin{theorem}[{\cite[Theorem 2]{HRW18}}] Consider the gluing $X = M_1 \cup_h M_2$, where the $M_i$ are compact, oriented 3-manifolds with torus boundary and $h: \partial M_2 \rightarrow \partial M_1$ is an orientation reversing homeomorphism for which $h(z_2) = z_1$. Then
\[
\widehat{\textit{HF}}(X) \cong \textit{HF}(\widehat{\textit{HF}}(M_1), h(\widehat{\textit{HF}}(M_2))),
\]
where intersection Floer homology is computed in $T_{M_1}$ and the isomorphism is one of relatively graded vector spaces that respects the $\text{Spin}^{c}$ decomposition. 
\label{thm:ICpairing}
\end{theorem}

Let $\pi: \text{Spin}^{c}(X) \rightarrow \text{Spin}^{c}(M_1) \times \text{Spin}^{c}(M_2)$ be the projection. Here respecting $\text{Spin}^{c}$ decomposition means that $\widehat{\textit{HF}}(X, \mathfrak{t})$ is a summand of $\textit{HF}(\widehat{\textit{HF}}(M_1, \s_1), h(\widehat{\textit{HF}}(M_2, \s_2)))$ only for $\mathfrak{t} \in \pi^{-1}(\s_1 \times \s_2)$. Pairing takes place in $T_{M_1}$, which can be cramped if the invariants $\widehat{\textit{HF}}(M_i)$ are sufficiently involved. We will almost always find it more convenient to carry out the intersection Floer homology in specific covers of $T_{M_1}$.

\begin{definition}[{\cite{HRW18}}]
Let $\overline{T}_M$ be the cover of $T_M$ associated with the kernel of the composition $\pi_1(T_M) \rightarrow \pi_1(\partial M) \rightarrow H_1(\partial M) \rightarrow H_1(M)$, and $p: \overline{T}_M \rightarrow T_M$ the projection. Further, let $\widetilde{T}$ denote the cover $\R^2 / (\frac{1}{2} + \Z)^2$.
\end{definition}

The lifts of $\widehat{\textit{HF}}(M)$ to these covers are useful for identifying properties that the invariant encodes, in addition to simplifying homology computations. These lifts decompose as
\[
\widehat{\textit{HF}}(M) \cong \bigoplus_{\s \in \text{Spin}^{c}(M)} p(\widehat{\textit{HF}}(M, \s)),
\]
where $\widehat{\textit{HF}}(M, \s)$ denotes the lift to $\overline{T}_{M, \s}$ of the part of $\widehat{\textit{HF}}(M)$ associated to $\s$ \cite[Theorem 7]{HRW16}. This lift $\widehat{\textit{HF}}(M, \s)$ is well-defined up to action by the deck group of $p$, and will typically be taken to be centered about the origin in $\overline{T}_{M, \s}$. When $H_1(M) \cong \Z$, which is encountered in the case of $S^3$ knot complements, the cover $\overline{T}_M$ may be identified with the infinite cylinder $S^1 \times (\R \setminus (\frac{1}{2} + \Z))$, with lifts $\overline{z}_i$ of the basepoint at coordinates $(0, i+\frac{1}{2})$. As an example, Theorem \ref{thm:ICpairing} applied to Dehn surgery shows the convenience of working with immersed curves. In this setting, we have 
\[
\widehat{\textit{HF}}(S^3_r(K)) \cong \textit{HF}(\widehat{\textit{HF}}(S^3 \setminus \nu J), h(\widehat{\textit{HF}}(D^2 \times S^1))),
\]
where $h$ is the slope $r$ surgery map. Figure \ref{fig:4surgeryT25example} shows $\widehat{\textit{HF}}(S^3_4(T(2,5)))$, where the four $\text{spin}^{c}$ structures are in correspondence with the four lifts of $h(\widehat{\textit{HF}}(D^2 \times S^1))$ required to lift all intersections to $\overline{T}_M$. It is immediate that this manifold is an L-space when lifting to $\overline{T}_M$, showing the advantage of working in the cylindrical cover.

\begin{figure}[!h]
\centering
\includegraphics[scale=1]{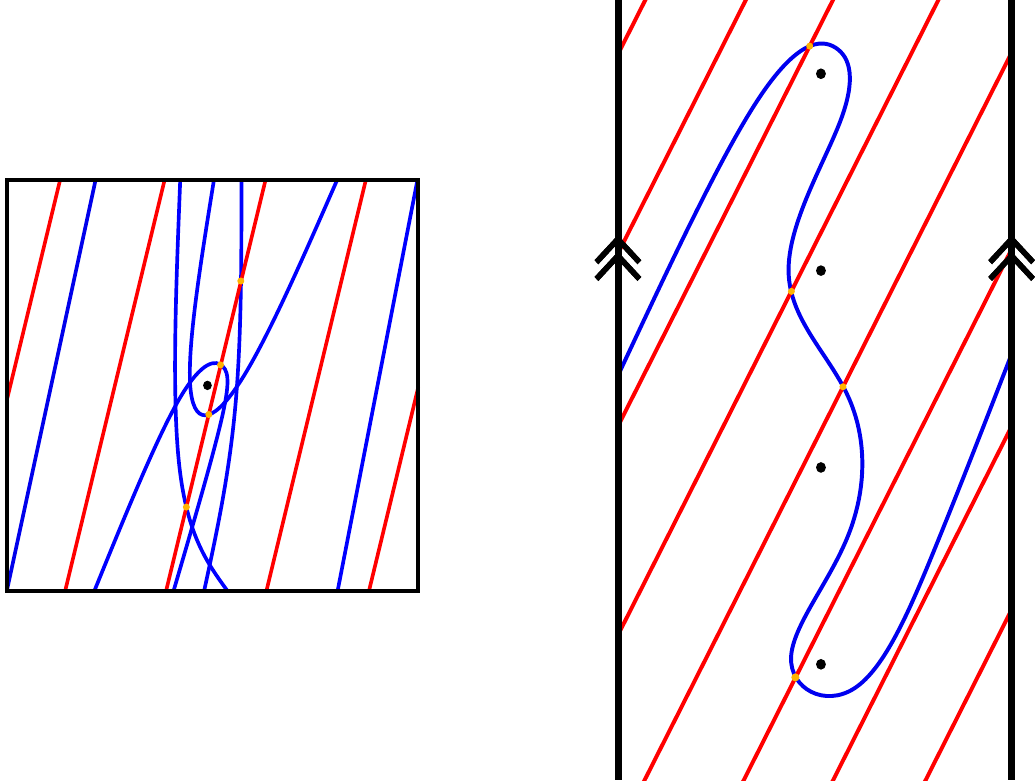}
\caption{The pairing of $\widehat{\textit{HF}}(S^3 \setminus \nu T(2,5))$ and $h(\widehat{\textit{HF}}(D^2 \times S^1))$ that computes $\widehat{\textit{HF}}(S^3_4(T(2,5))$.} 
\label{fig:4surgeryT25example}
\end{figure}

\subsection{Structure and properties of $\widehat{\textit{HF}}(M)$}

The immersed curves invariant $\widehat{\textit{HF}}(M)$ encodes many algebraic properties as geometric symmetries. For example, orientation reversal yields $\widehat{\textit{HF}}(-M) = \widehat{\textit{HF}}(M)$, but the boundary parameterization changes. Seen in $T_M$, the invariant $\widehat{\textit{HF}}(-M)$ is obtained by reflecting about the rational longitude. The invariant also carries information about Turaev torsion and the Thurston norm for knot complements \cite{HRW18}.

A structural property of immersed curves is their invariance (as unlabelled curves) under the action by the elliptic involution, with $z$ fixed, of $\partial M \setminus z$ \cite[Theorem 7]{HRW18}. More concretely, Hanselman, Rasmussen, and Watson show that $\widehat{\textit{CFD}}(M, c(\s)) \cong \mathbf{E} \boxtimes \widehat{\textit{CFD}}(M, \s)$, meaning that $\text{Spin}^{c}$ conjugation on the level of bordered invariants achieves the same resulting curve invariant as what would arise from the box tensor product with a particular type DA structure associated with elliptic involution. This feature of $\widehat{\textit{HF}}(M)$ can be seen when a rotation of $\pi$ about the origin is applied to our chosen lift $\widehat{\textit{HF}}(M, \s)$ in $\overline{T}_M$. 

We may also arrange the curves of $\widehat{\textit{HF}}(M)$ to ensure that the intersection Floer homology of pairings is minimal. 

\begin{definition}
Fix a metric on the torus $T_M$. We say $\widehat{\textit{HF}}(M)$ is in pegboard form if the immersed curves are homotoped to have minimal length in $T_M$, where the curves remain outside an $\epsilon$-ball of $z$.
\label{def:pegboard}
\end{definition}

The result is a pegboard representative for $\widehat{\textit{HF}}(M)$, and we may lift these to both $\overline{T}_M$ and $\widetilde{T}$, where each lift of $z$ has an $\epsilon$-ball disjoint from the lift(s) of $\widehat{\textit{HF}}(M)$. Pegboard forms are invaluable for pairing, since pulling curves tight homotopes away pseudo-holomorphic disks that do not contribute to the intersection Floer homology of a pairing. This ensures that the resulting Floer homology is minimal \cite[Lemma 47]{HRW16}. 

\begin{figure}[!h]
\centering
\def\svgscale{0.7}
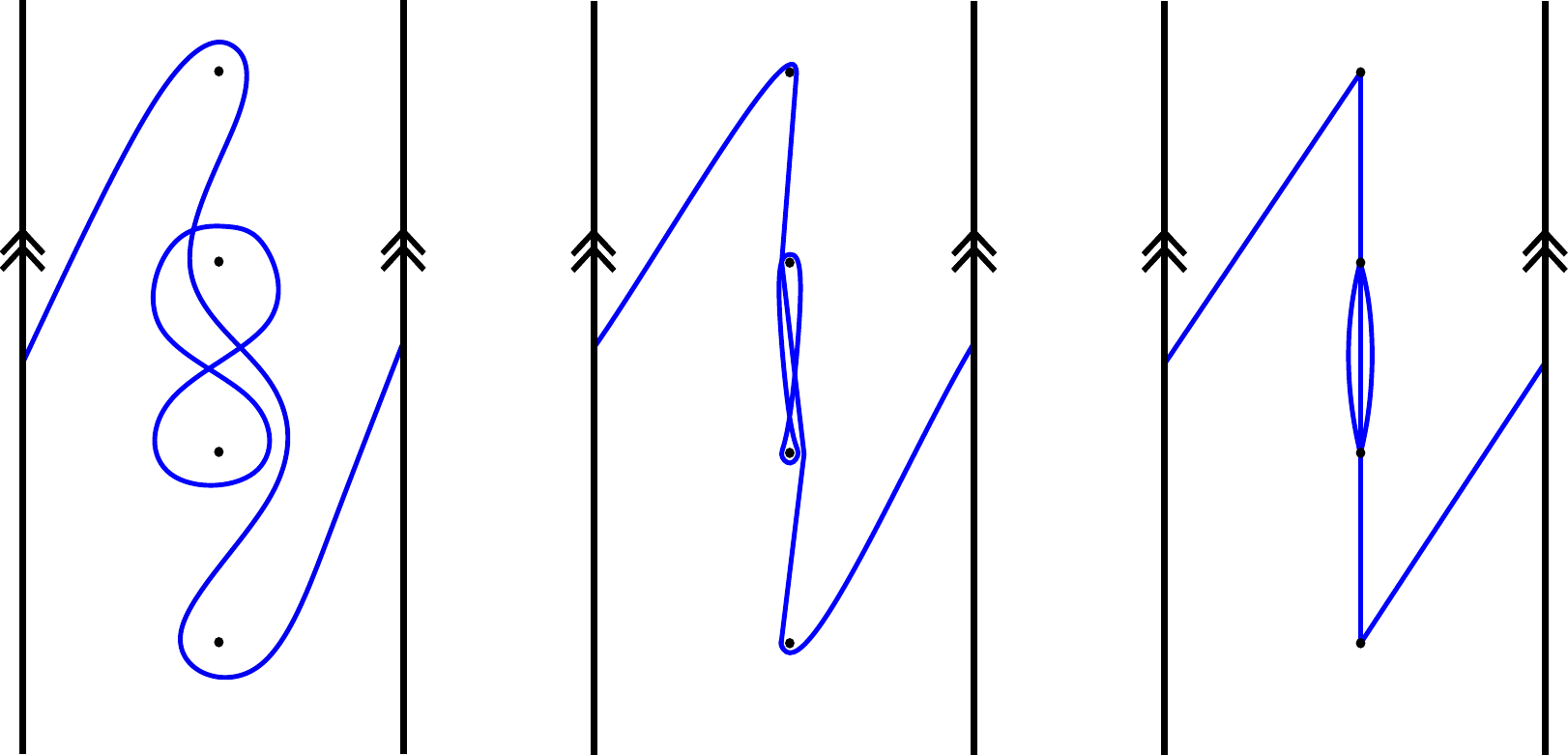
\caption{Pulling the immersed curve invariant $\widehat{\textit{HF}}(S^3 \setminus (T(2,3) \# T(2,3)))$ (without grading arrows) tight to pegboard form. The figures from left to right show the stages of homotoping the invariant to lie within a neighborhood of the lifts of the meridian $\mu$.} 
\label{fig:SumRHTGod}
\end{figure}

If $H_1(M) \cong \Z$ so that $\overline{T}_M \cong S^1 \times (\R \setminus \frac{1}{2} \Z)$, define $n_i$ to be the number of vertical segments of the pegboard representative of $\widehat{\textit{HF}}(M, \s)$ that are parallel to $\overline{\mu}_i$, the lift of $\mu$ at height $i$. When $M = S^3 \setminus \nu J$, the lift $\widehat{\textit{HF}}(M, \s)$ satisfies a conveniently simple form. The lift $\widehat{\textit{HF}}(M, \s)$ consists of inessential curves, or curves that are null-homotopic after allowing homotopies through the basepoints, and a single essential curve $\overline{\gamma}$ that is homotopic to the homological longitude when allowing homotopies through the basepoints \cite[Corollary 63]{HRW18}. 

The pegboard representative of $\widehat{\textit{HF}}(M, \s)$ also encodes some numerical and concordance invariants of $J$. For example, the genus is given by height (rounded up to the nearest integer) of the tallest $\overline{z}_s$ around which the pegboard representative of $\widehat{\textit{HF}}(M, \s)$ wraps in $\overline{T}_M$. Additionally, the height (rounded to the largest integer) of the first $\overline{z}_s$ around which $\overline{\gamma}$ wraps is precisely the Oszv\'ath-Szab\'o invariant $\tau(J)$. Hom's $\epsilon$ invariant may also be determined from $\widehat{\textit{HF}}(M, \s)$ by observing $\overline{\gamma}$ after it wraps around this basepoint: the curve turns downwards, upwards, or continues straight corresponding to $\epsilon(J)$ being 1, -1, and 0, respectively. Notice that the essential curve $\overline{\gamma}$ only continues straight if $\tau(J) = 0$. These two invariants determine the slope of the curve $\overline{\gamma}$, given by $2\tau(J) - \epsilon(J)$.

The invariance of $\widehat{\textit{HF}}(M)$ under the action of the hyperelliptic involution implies that $n_{-i} = n_i$ for its pegboard representative. It will be particularly useful to characterize those knots whose complements have curve invariants with minimal $n_i$ for all $i \in \Z$. Recall that a knot $J$ is an \textit{L-space knot} if it admits an L-space surgery, and that it also satisfies $g(J) = |\tau(J)|$. The following lemma is surely known to experts of the field, but is included here for completeness and to serve as an example of the types of computations with immersed curves to come.

\begin{lemma}
Let $M = S^3 \setminus \nu J$ with $J$ non-trivial. Then $J$ is an $L$-space knot if and only if $\widehat{\textit{HF}}(M, \s)$ pulls tight to a curve with $n_i = 1$ for $|i| < g(J)$ and $n_i = 0$ otherwise.
\label{lem:lspaceJ}
\end{lemma}

\begin{proof}
When $J$ is a genus $g$ knot with an $L$-space surgery $S^3_p(J)$, by mirroring if necessary we may take $p$ to be positive. A surgery exact triangle argument shows that $S^3_{p+1}(J)$ is an $L$-space, and likewise for $S^3_k(J)$ with integral $k > p$. For some $k > 2g-1$, Theorem \ref{thm:largesurgery} then additionally provides that 
\[
\widehat{\textit{HF}}(S^3_k(J), [s]) \cong \widehat{A}_s
\]
for all $s \in \Z$. Then each $\widehat{A}_s \cong \F$ since $S^3_k(J)$ is an L-space. We can view $S^3_k(J)$ as the $+k$-sloped gluing of $D^2 \times S^1$ to $M$, so that Theorem \ref{thm:ICpairing} guarantees
\[
\widehat{\textit{HF}}(S^3_k(J)) \cong \textit{HF}(\widehat{\textit{HF}}(M), h(\widehat{\textit{HF}}(D^2 \times S^1))).
\]

Analogous to the $S^3_4(T(2,5))$ example, precisely $k$ lifts of the $+k$-sloped curve $\widehat{\textit{HF}}(D^2 \times S^1)$ are required to lift all intersections in $T_M$ to $\overline{T}_M$, and each lift is in correspondence to precisely one $\text{spin}^{c}$ structure of $\text{Spin}^{c}(S^3_k(J))$. These differ in height by one in $\overline{T}_M$, and each lift must intersect the essential curve $\overline{\gamma}$ at least once. An inessential curve component contributes an even number of vertical segments to some $n_i$ with $|i| < g(J)$. If the pegboard representative of $\widehat{\textit{HF}}(M)$ contains such a component, then dim $\widehat{\textit{HF}}(S^3_k(J),[s]) > 1$ for some $\text{spin}^{c}$ structure $[s]$ since $k > 2g(J)-1$ (we are guaranteed that some lift of $h(\widehat{\textit{HF}}(D^2 \times S^1)$ crosses $\overline{\mu}_i$). As $S^3_k(J)$ is an $L$-space, we must have not have any inessential curve components. Then $\overline{\gamma}$ is the only component of $\widehat{\textit{HF}}(S^3 \setminus \nu J)$, and so $n_i = 1$ for $|i| < g(J)$ and $n_i = 0$ otherwise.

If $\widehat{\textit{HF}}(M, \s)$ has a pegboard representative satisfying $n_i = 1$ for $|i| < g(J)$ and $n_i = 0$ otherwise, then by mirroring if necessary we may suppose $\tau(J) > 0$ since $J$ is non-trivial. The invariant $\widehat{\textit{HF}}(M, \s)$ is just $\overline{\gamma}$ since no $n_i$ has room to admit an inessential component. Further, $\overline{\gamma}$ has slope $2\tau(J) - \epsilon(J) = 2g(J)-1$ and pulls tight to vertical segments parallel to $\overline{\mu_i}$ for $|i| < g(J)$. If $h$ is a gluing with slope $k > 2g(J)-1$, then each lift of $h(\widehat{\textit{HF}}(D^2 \times S^1))$ intersects $\widehat{\textit{HF}}(M, \s)$ at most once. Using Theorem \ref{thm:ICpairing}, we have dim $\widehat{\textit{HF}}(S^3_k(J), [s]) = 1$ for each $\text{spin}^{c}$ structure $[s]$. Therefore $S^3_k(J)$ is an $L$-space, and so $J$ is an $L$-space knot.
\end{proof}

\noindent\textit{Remark.}
For an $L$-space knot $J$, we have $|\tau(J)| = g(J)$ and so the pegboard representative of $\widehat{\textit{HF}}(M)$ takes on one of two mirrored forms depending on the sign of $\tau(J)$. These are illustrated in Figure \ref{fig:Lspaceknot} for a genus two knot.

\begin{figure}[!ht]
\centering
\def\svgscale{0.8}
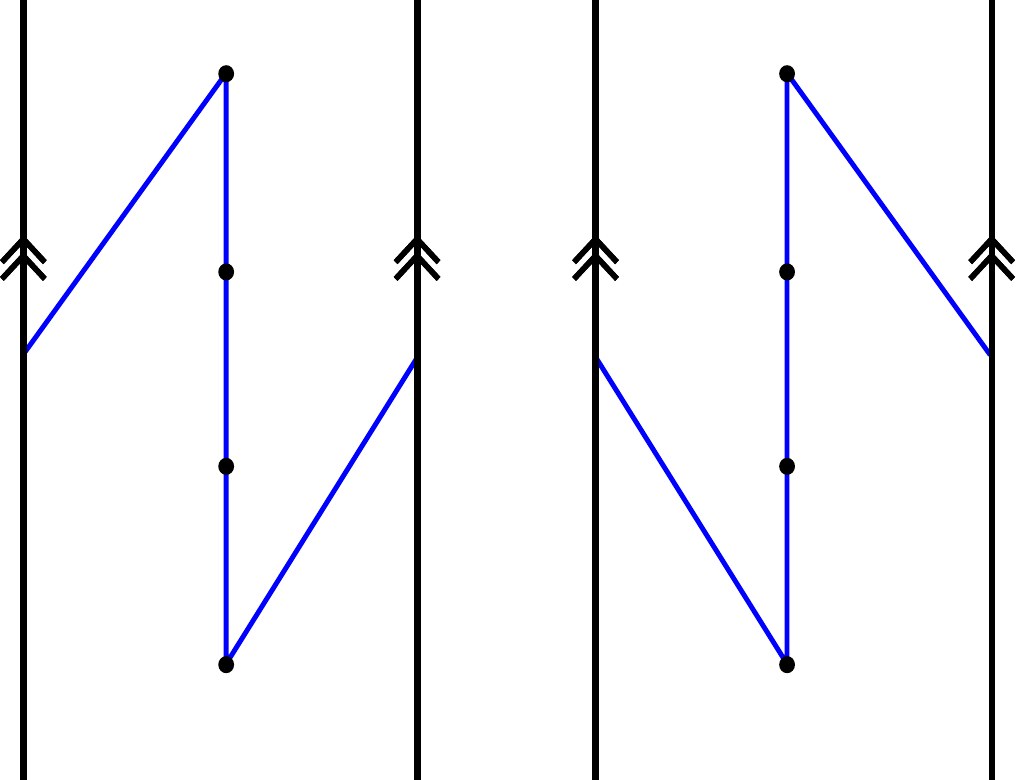
\caption{The two curve invariants for a genus two L-space knot, placed in pegboard form. Left: $\tau(J)=g(J)$. Right: $\tau(J)=-g(J)$.}
\vspace{-1.5\intextsep}
\label{fig:Lspaceknot}
\end{figure}

\section{Pairings and immersed curves}
\label{sec:proof}

As alluded to in the introduction, if a closed, orientable $3$-manifold $X$ contains a Klein bottle, then we may view $X$ as a gluing $X = M \cup_h N$ of a rational homology solid torus $M$ and the twisted $I$-bundle over the Klein bottle $N$. Alternatively we can view $M = Y \setminus \nu J$ as a knot manifold, which is the complement of a knot $J$ in $Y$ some rational homology sphere. We specialize to gluings of $N$ to $S^3$ knot complements, and use the immersed curves machinery to determine which knot complements can glue with $N$ to have the Floer homology as that of $S^3_8(K)$ with $g(K)=2$. The proof of Theorem \ref{thm:main} involves establishing three main lemmas that imply the following:

\begin{itemize}
\setlength\itemsep{-0.8em}
\item If $J$ is the unknot, then $X = S^3_8(T(2,5))$. \\
\item If $J$ is non-trivial, then dim $\widehat{\textit{HF}}(X)$ forces $J = T(2,3)$. \\
\item If $J = T(2,3)$, then $X$ does not arise as $S^3_8(K)$ with $g(K)=2$.
\end{itemize}

The first lemma does not require immersed curves, but does use Maslov grading information for $\widehat{\textit{HF}}$ via the $d$-invariants. The second lemma is where immersed curves are used to constrain $J$. Finally the third lemma, in Section \ref{sec:gr}, uses the relative $\Q$-grading on the box tensor product of bordered invariants for $S^3 \setminus \nu T(2,3)$ and $N$.

\subsection{Pairings}
\label{sub:pairings}

Let us now pin down the possibilities for $h$. Summands of homology will typically be ordered with the summand generated by the rational longitude first, such as in $H_1(\partial N) \cong \Z_{[\phi_0]} \oplus \Z_{[\phi_1]}$.
 
\begin{definition}
Let $X = (S^3 \setminus \nu J) \, \cup_h N$ be the gluing of $N$ to the complement $S^3 \setminus \nu J$, where the orientation-reversing gluing induces $h_{\ast}$ on homology given by
\begin{equation*}
[h_{\ast}] = 
\begin{pmatrix}
q & r \\
p & s \\
\end{pmatrix}.
\end{equation*}
We say $h$ is a slope $p/q$ gluing, corresponding to the slope of $h_{\ast}(\phi_0)$.
\label{def:X}
\end{definition}

\begin{proposition}
Let $h$ be defined as above. Then $|H_1(X)|=8$ if only if $|p|=2$. Additionally, we have
\begin{align*}
H_1(X) = \left\{
	\begin{array}{l c}
	\Z/2\Z \oplus \Z/4\Z & s \equiv 0 \,\, (\text{mod} \,\, 2) \\
	\Z/8\Z & s \not\equiv 0 \,\, (\text{mod} \,\, 2)
	\end{array}
	\right.
\end{align*}
\label{prop:pairhomology}
\end{proposition}

\begin{proof}
Recall the parameterization on $\partial N$ by the rational longitude $\phi_0$ and our chosen dual curve $\phi_1$, so that $H_1(\partial N) \cong \Z_{[\phi_0]} \oplus \Z_{[\phi_1]}$. As a slight abuse of notation, let $\phi_0$ and $\phi_1$ also denote the inclusion of these slopes in $H_1(N)$. The dual curve $\phi_1$ includes in $H_1(N)$ as twice some primitive curve $x$ since $\phi_0$ includes with order two \cite[Subsection 3.1]{Wat12}, and so $H_1(N) \cong \Z/2\Z_{[\phi_0]} \oplus \Z_{[x]}$. For the knot complement, $H_1(\partial(S^3 \setminus \nu J)) \cong \Z_{[\lambda]} \oplus \Z_{[\mu]}$ and $H_1(S^3 \setminus \nu J) \cong \Z_{[\mu]}$, with the inclusions $[\mu]$ primitive and $[\lambda]$ trivial. Then
\begin{align*}
H_1(X) \cong \faktor{(H_1(S^3 \setminus \nu J) \oplus H_1(N))}{f_{\ast}(H_1(\partial N))},
\end{align*}
where $f_{\ast}$ maps $H_1(\partial N)$ into $H_1(N)$ by inclusion and into $H_1(S^3 \setminus \nu J)$ through $h_{\ast}$ and inclusion. 

The quotient identifies $\phi_0 \sim q\lambda + p\mu$ and $\phi_1 = 2x \sim r\lambda + s\mu$, and so $H_1(X)$ has the following presentation:
\begin{align*}
H_1(X) &\cong \langle \lambda, \mu, \phi_0, x \,\, \vert \,\, \lambda = 0, 2\phi_0 = 0, \phi_0 = q\lambda + p \mu, 2x = r\lambda + s\mu \rangle \\
&\cong \langle \mu, x \,\, \vert \,\, 2p\mu = 0, 2x = s\mu \rangle.
\end{align*}
From this we see that $|H_1(X)| = 8$ if and only if $|p|=2$, and that $H_1(X)$ is cyclic when $s \not\equiv 0 \,\, (\text{mod} \,\, 2)$.
\end{proof}
A gluing $h$ that satisfies the cyclic condition of Proposition \ref{prop:pairhomology} will be referred to as a \textit{cyclic gluing}.

\subsection{Dehn twisting invariance}
\label{subsec:dehntwist}

A \textit{Heegaard Floer homology solid torus} $M$ is a rational homology solid torus satisfying 
\[
\widehat{\textit{CFD}}(M, \mu_M, \lambda_M) \cong \widehat{\textit{CFD}}(M, \mu_M + \lambda_M, \lambda_M),
\]
with $\lambda_M$ the rational longitude of $M$ and $\mu_M$ any slope dual to $\lambda$. The twisted $I$-bundle over the Klein bottle is shown to be a Heegaard Floer homology solid torus in [BGW, Proposition 7], and this invariance for $A(\theta_{N, \s_1})$ is shown in Figure \ref{fig:N1twistinvariance}. The new yellow edge recovering the $\rho_3$ edge is the result of applying  edge reduction \cite[Section 2.6]{Lev14}. We remark that the case for the loose component $A(\theta_{N, \s_0})$ is immediate. Equivalently, inspection of $\widehat{\textit{HF}}(N)$ in Figure \ref{fig:N1and2Abase} reveals that the curve invariant can be homotoped (without crossing the basepoint) to lie within a neighborhood of the rational longitude.

\begin{figure}[!ht]
\centering
\includegraphics[scale=0.6]{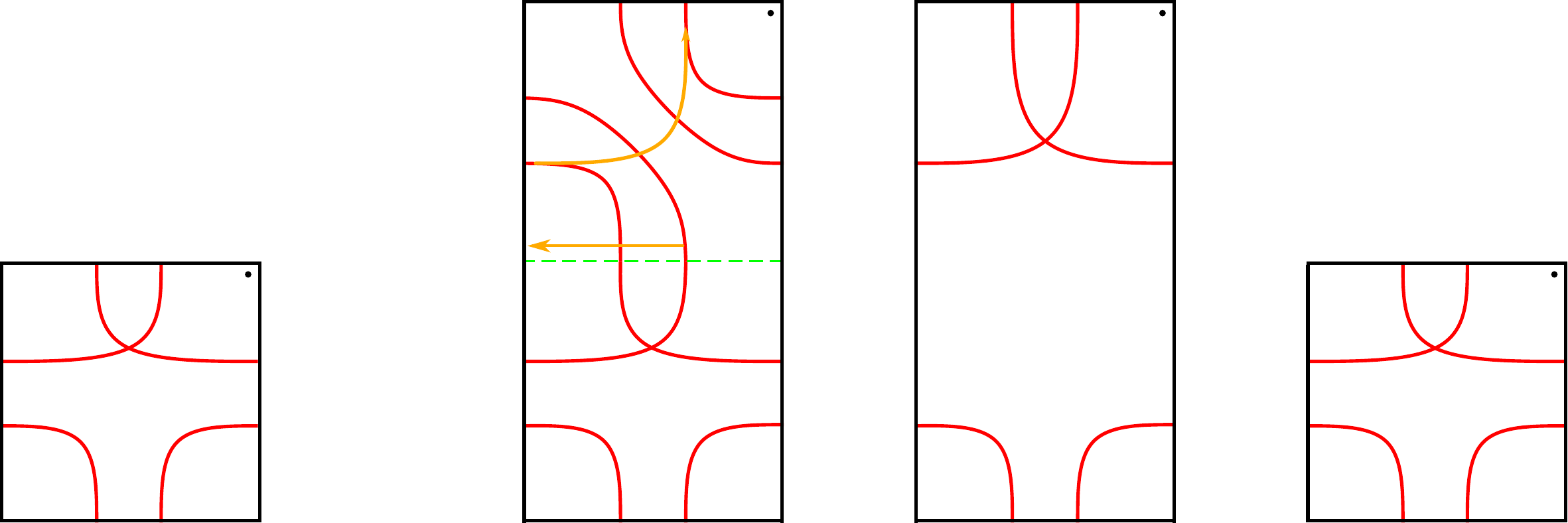}
\caption{Left: The type A realization of $\widehat{\textit{CFD}}(N, \phi_1, \phi_0, \s_1)$. Right: Dehn twisting to obtain the type A realization of $\widehat{\textit{CFD}}(N, \phi_1+\phi_0, \phi_0, \s_1)$, and the edge reduction showing homotopy equivalence.}
\label{fig:N1twistinvariance}
\end{figure}

Dehn twisting $n$ times along $\phi_1$, and then gluing is equivalent to pre-composing $[h_{\ast}]$ with
\[
[T_n] = \begin{pmatrix}
1 & n \\
0 & 1 \\
\end{pmatrix}, \,\, \text{yielding} \,\, 
[h_{\ast} \circ T_n] = \begin{pmatrix}
q & r+nq \\
p & s+np \\
\end{pmatrix}.
\]
Gluing by either map yields manifolds with equivalent ranks of $\widehat{\textit{HF}}(X, \mathfrak{t})$ for each $\text{spin}^{c}$ structure [HRW2, Corollary 27]. The mod p residue class of $[s]$ is preserved for all $n \in \Z$, and so we have to consider an integral family of manifolds $X$ obtained by varying $n$. This will be more pertinent in the next section where we appeal to more than just dim $\widehat{\textit{HF}}(X)$, so let us then restrict attention to the maps $h$ with $0 \leq s < 2$. We are also interested in cyclic gluings, which means $s \equiv 1 \,\, (\text{mod} \,\, 2)$ due to Proposition \ref{prop:pairhomology}. We will show that the pairing slope $\pm 2/q$ must be integral when $J$ is non-trivial to provide the right dim $\widehat{\textit{HF}}(X)$, so we will also take $q = 1$. The Dehn twisting invariance of $\widehat{\textit{CFD}}(N)$ allows us to choose $s = -1$, and so the prototypical gluing $h$ then induces
\[
[h_{\ast}] = \begin{pmatrix}
1 & 0 \\
2 & -1 \\
\end{pmatrix}.
\]

\subsection{Gluings with $J$ trivial}
\label{subsec:Jtrivial}

We now establish the first lemma handling the case when $J$ is trivial, which just involves Dehn fillings of $N$, before approaching the case when $J$ is non-trivial. Continuing as before, let $M$ will denote the knot complement $M = S^3 \setminus \nu J$.

\begin{lemma}
Suppose $X = (Y \setminus \nu J) \cup_h N$ contains a Klein bottle and is realized as $S^3_8(K)$ with $g(K) = 2$. If $J$ is trivial, then $Y=S^3$, $K=T(2,5)$, and $X = (-1; \frac{1}{2}, \frac{1}{2}, \frac{2}{5})$ as a Seifert fibered manifold.
\label{lem:MLemFinite}
\end{lemma}

\begin{proof}
We know from Corollary \ref{cor:irreduciblecomplement} that $M$ is irreducible, and so $Y = S^3$ and $M = D^2 \times S^1$. Such a gluing $X$ is a Dehn filling $N(\alpha)$, where $\alpha$ is a slope on $\partial N$. The twisted $I$-bundle over the Klein bottle $N$ has a Seifert structure with base orbifold $D^2(2,2)$ \cite{LW14}, and we may parametrize $\partial N$ using $\{\phi_0,\phi_1\}$ as before, where $\phi_0 = \lambda_N$ is the rational longitude of $N$ and $\phi_1$ is our preferred choice of curve dual to $\phi_0$. We have $N(\phi_1) = \R P^3 \# \R P^3$ and $b_1(N(\phi_0)) > 0$, and so we can consider $\alpha \neq \phi_0, \phi_1$. Since $N$ is a Heegaard Floer homology solid torus, $N(\alpha)$ is an L-space for all $\alpha \neq \phi_0$ \cite[Theorem 26]{HRW18}. Thus, $\widehat{\textit{HFK}}(K) \cong \widehat{\textit{HFK}}(T(2,5))$ since $K$ required to be an L-space knot.

Any Dehn filling $N(\alpha)$ for which $\alpha \neq \phi_0, \phi_1$ admits a pair of Seifert structures with base orbifolds $\R P^2(\Delta(\alpha, \phi_0))$ and $S^2(2,2, \Delta(\alpha, \phi_1))$. The filling $N(\alpha)$ has $b_1(N(\alpha))=0$ since $\partial N$ compresses in $S^3 \setminus \nu J$, and so it is either a lens space or non-cyclic. We obstruct $X \cong L(8,q)$ using the $d$-invariants, also known as the Heegaard Floer correction terms, from \cite{OS03a}. 

The $d$-invariants $d(L(p,q), [s])$ for lens spaces are well known and may be computed recursively using \cite[Proposition 4.8]{OS03a}, together with $d(-Y, [s]) = -d(Y, [s])$ and $d(L(1,1), [0]) = d(S^3, [0]) = 0$. In \cite{NW15} the $d$-invariants of integral surgeries are related to those of $d(L(p,1), [s])$ by
\[
d(S^3_p(K), [s]) = d(L(p,1), [s]) - 2\max\{V_s(K), H_{s-p}(K)\},
\]
where the $V_s(K)$ and $H_s(K)$ are integers related to the quotient complexes $\mathcal{A}^+_s$ of $CFK^{\infty}(S^3, K)$. Since $K$ has isomorphic knot Floer homology to that of $T(2,5)$, we must have $V_s(K) = V_s(T(2,5))$ and $H_s(K) = H_s(T(2,5))$. With $V_0(K)=V_1(K)=1$ and $V_s(K)=0$ for $s \geq 2$, one computes
\begin{align*}
d(S^3_8(K),[s]) = \left\{
	\begin{array}{l c}
	 -1/8 & s \equiv 5 \,\, (\text{mod} \,\, 8) \\
	 1/4 & s \equiv 6 \,\, (\text{mod} \,\, 8) \\
	 -9/8 & s \equiv 7 \,\, (\text{mod} \,\, 8) \\
	 -1/4 & s \equiv 0 \,\, (\text{mod} \,\, 8) \\
	 -9/8 & s \equiv 1 \,\, (\text{mod} \,\, 8) \\
	 1/4 & s \equiv 2 \,\, (\text{mod} \,\, 8) \\
	 -1/8 & s \equiv 3 \,\, (\text{mod} \,\, 8) \\
	 -1/4 & s \equiv 4 \,\, (\text{mod} \,\, 8)
	\end{array}
	\right.
\end{align*}

\noindent We have $d(L(8, \pm1),[t_1])= \pm 7/4$ and $d(L(8, \pm 3), [t_2]) = \pm 5/8$ for some $t_1,t_2 \in \Z$ from the recursion in \cite[Proposition 4.8]{OS03a}, but both of these differ from any $d(S^3_8(K), [s])$. Therefore, $X$ cannot be a lens space.

In \cite[Theorem 2]{Doi15}, Doig classifies finite, non-cyclic surgeries $S^3_r(K)$ for $|r| \leq 9$. Among these, the manifolds with $|H_1(X)| = 8$ are specific dihedral manifolds that are small Seifert fibered with base orbifold $S^2$. They are $-S^3_8(T(2,3)) = \left(-1; \tfrac{1}{2}, \tfrac{1}{2}, \tfrac{2}{3} \right)$, and $S^3_8(K) = \left(-1; \tfrac{1}{2}, \tfrac{1}{2}, \tfrac{2}{5} \right)$ for $\widehat{\textit{HFK}}(K) \cong \widehat{\textit{HFK}}(T(2,5))$. However, since eight is a characterizing slope for $T(2,5)$ due to \cite{NZ18}, we have $\left(-1; \tfrac{1}{2}, \tfrac{1}{2}, \tfrac{2}{5} \right) = S^3_8(T(2,5))$. For $T(2,3)$ we have $V_1(T(2,3)) = 0$, and so 
\[
d(-S^3_8(T(2,3), [1]) = -d(S^3_8(T(2,3)), [1]) = -\tfrac{7}{8}.
\]
This differs from any $d(S^3_8(K), [s])$ above when $K$ has the same knot Floer homology as $T(2,5)$, establishing the lemma.
\end{proof}

\subsection{Gluings with $J$ non-trivial}
\label{subsec:Jnontrivial}

Now suppose that $J$ is a non-trivial knot in $S^3$. There is a convenient visual way to track Floer homology associated to a given $\text{spin}^{c}$ structure of $X$, that we have already encountered in both Figure \ref{fig:4surgeryT25example} and in proving Lemma \ref{lem:lspaceJ}.

\begin{wrapfigure}{r}{0.4\linewidth}
\centering
\def\svgscale{0.6}
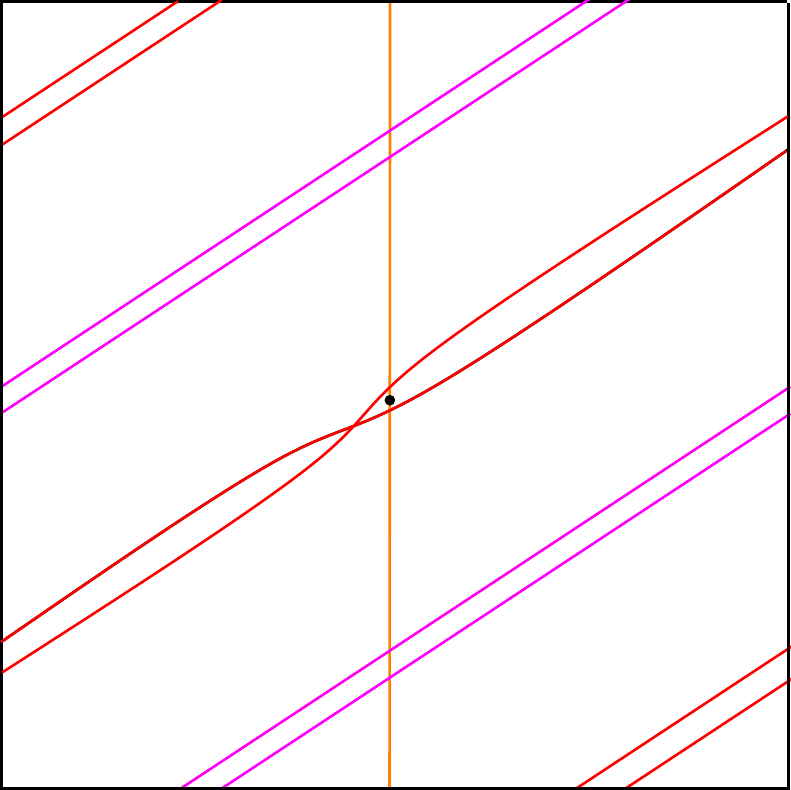
\caption{The 2/3-sloped curves of $N$ associated to $\s_1$ in $T_M$.}
\vspace{-1.5\intextsep}
\label{fig:NTwoThirds}
\end{wrapfigure}

By Theorem \ref{thm:ICpairing}, we have $\widehat{\textit{HF}}(X, \mathfrak{t})$ isomorphic to a summand of $HF(\widehat{\textit{HF}}(M, \s), h(\widehat{\textit{HF}}(N), \s_k))$ when $\mathfrak{t} \in \pi^{-1}(\s \times \s_k)$. Intersections $\mathbf{x}, \mathbf{y}$ in pairing generate Floer homology in the same $\text{spin}^{c}$ structure if and only if there exist paths $p_0$ from $\mathbf{x}$ to $\mathbf{y}$ in $\widehat{\textit{HF}}(M, \s)$ and $p_1$ from $\mathbf{x}$ to $\mathbf{y}$ in $h(\widehat{\textit{HF}}(N, \s_k))$, such that the concatenation of $p_0$ with $-p_1$ lifts to a closed, piecewise smooth path in $\widetilde{T}$ \cite[Section 2]{HRW18}.  When $h$ has slope $\pm 2/q$, a single lift $h(\widehat{\textit{HF}}(N, \s_k))$ of the component of $h(\widehat{\textit{HF}}(N))$ corresponding to $\s_k$ will fail to lift all intersections in $T_M$ generated by this component; two lifts of the component are required. 

This is motivated by Figure \ref{fig:NTwoThirds} for $h(\widehat{\textit{HF}}(N, \s_1))$ projected to $T_M$, where $h$ a cyclic gluing of slope 2/3. A single lift to $\overline{T}_M$ cannot simultaneously lift all intersections by the suggestively colored red and purple pieces of the projection of $h(\widehat{\textit{HF}}(N, \s_1)$. A simpler example of this is seen from Figure \ref{fig:4surgeryT25example}, where four lifts of $h(\widehat{\textit{HF}}(D^2 \times S^1)$ are needed. Additionally, Figure \ref{fig:plus2RHTpair} provides an example of the four required curves of $h(\widehat{\textit{HF}}(N, \s_1))$ with $h$ a cyclic gluing of slope 2. Notice that the lifted curves cannot generate intersections with the same $\text{Spin}^{c}$ grading as there is no path between the curves in $h(\widehat{\textit{HF}}(N, \s_k))$. For this reason, we can associate the eight $\mathfrak{t} \in \text{Spin}^{c}(X)$ with these lifted curves of $h(\widehat{\textit{HF}}(N, \s_1))$ and $h(\widehat{\textit{HF}}(N, \s_0))$.

Now that we can distinguish intersections generating Floer homology in different $\text{Spin}^{c}$ structures, we establish the second lemma. This is the extent to which ignoring the Maslov grading in the immersed curves package can push the case where $J$ is non-trivial. Since we are only interested in $X$ with $H_1(X) \cong \Z/8\Z$ and constrained dim $\widehat{\textit{HF}}(X, \mathfrak{t})$ from Proposition \ref{prop:HF5of8}, we assume these traits.

\begin{lemma}
Let $J$ be non-trivial and consider $X = (S^3 \setminus \nu J) \cup_h N$, where $h$ is any slope 2/q cyclic gluing.
If rk $\widehat{\textit{HF}}(X, \mathfrak{t}) = 1$ for at least five $\mathfrak{t} \in \text{Spin}^{c}(X)$, then $h$ has slope 2 and $J=T(2, 3)$.
\label{lem:MLemS3}
\end{lemma}

\noindent\textit{Proof.}
Let $M = S^3 \setminus \nu J$, and pull $\widehat{\textit{HF}}(M, \s)$ tight to pegboard form. From Theorem \ref{thm:ICpairing} we have 
\[
\widehat{\textit{HF}}(X, \mathfrak{t}) \cong HF(\widehat{\textit{HF}}(M, \s), h^{\pm 2}(\widehat{\textit{HF}}(N), \s_k)),
\]
where $\mathfrak{t} \in \pi^{-1}(\s \times \s_k)$.

\begin{wrapfigure}{r}{0.45\linewidth}
\centering
\def\svgscale{0.4}
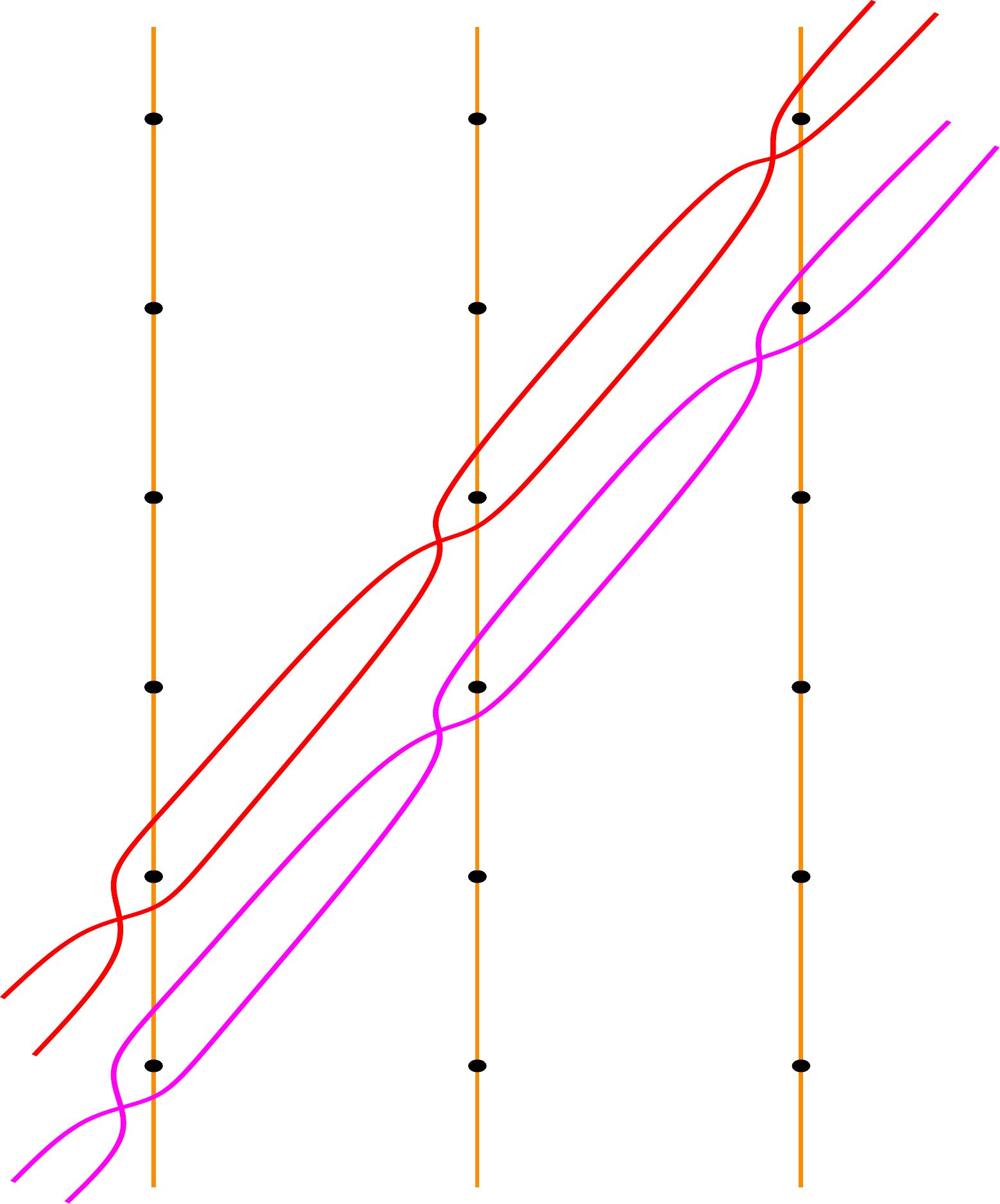
\caption{Intersections between the lifts of $h(\widehat{\textit{HF}}(N, \s_1))$ and potential vertical segments of $\widehat{\textit{HF}}(M)$, when $h$ is a slope 2 cyclic gluing.}
\vspace{-1.5\intextsep}
\label{fig:plus2modelt1}
\end{wrapfigure}

Suppose for the sake of contradiction that $|q| > 1$, so that the slope of $h$ satisfies $|2/q| < 1$. It is then immediate that all four lifts of the loose curves $h(\widehat{\textit{HF}}(N, \s_0))$ intersect each potential vertical segment of $\widehat{\textit{HF}}(M)$ more than once (such as in Figure \ref{fig:NTwoThirds}). Since dim $\widehat{\textit{HF}}(X, \mathfrak{t}) = 1$ for at least five $\mathfrak{t} \in \text{Spin}^{c}(X)$, we must have $n_i=0$ for all $|i| < g(J)$. However this condition is satisfied only by the unknot, which is the desired contradiction. Then $h$ is a slope 2 cyclic gluing.

Notice that each potential vertical segment of $\widehat{\textit{HF}}(M)$ at height $i$ intersects two of the four lifted curves of $h(\widehat{\textit{HF}}(N, \s_1))$, showcased in Figure \ref{fig:plus2modelt1}. Similarly, the same potential vertical segment intersects two of the four lifted curves of the loose component $h(\widehat{\textit{HF}}(N, \s_0))$. Then for all $i \in \Z$, we have $n_i$ contributing to rk $\widehat{\textit{HF}}(X, \mathfrak{t})$ in at least four different $\text{spin}^{c}$ structures. We are then forced to have $n_i \leq 1$ for all $i \in \Z$, which is equivalent to $\widehat{\textit{HF}}(M)$ not containing any inessential curve components. Therefore, $J$ is an $L$-space knot by Lemma \ref{lem:lspaceJ}. In the remark following that lemma, we see that the pegboard representative of $\widehat{\textit{HF}}(M) = \overline{\gamma}_{\s}$ for an L-space knot complement is completely determined by $\tau(J)$.

If $g(J) > 1$, four lifted components (two for each $\s_k$) of $h(\widehat{\textit{HF}}(N))$ intersect $\overline{\gamma}_{\s}$ more than once. This is shown in Figure \ref{fig:plus2tauJ2Lspacepair} when $\tau(J) = 2$, for two of the four such curves. Then we must have $g(J) = 1$ since $J$ is non-trivial, and so $J$ is either $T(2,3)$ or $T(2,-3)$. Observe that gluing $N$ to $S^3 \setminus \nu T(2,-3)$ by a $+2$-sloped cyclic gluing yields a manifold with excessive Floer homology since $\tau(T(2,-3)) < 0$, which establishes the lemma.
\qed

\begin{figure}[!ht]
\centering
\def\svgscale{0.5}
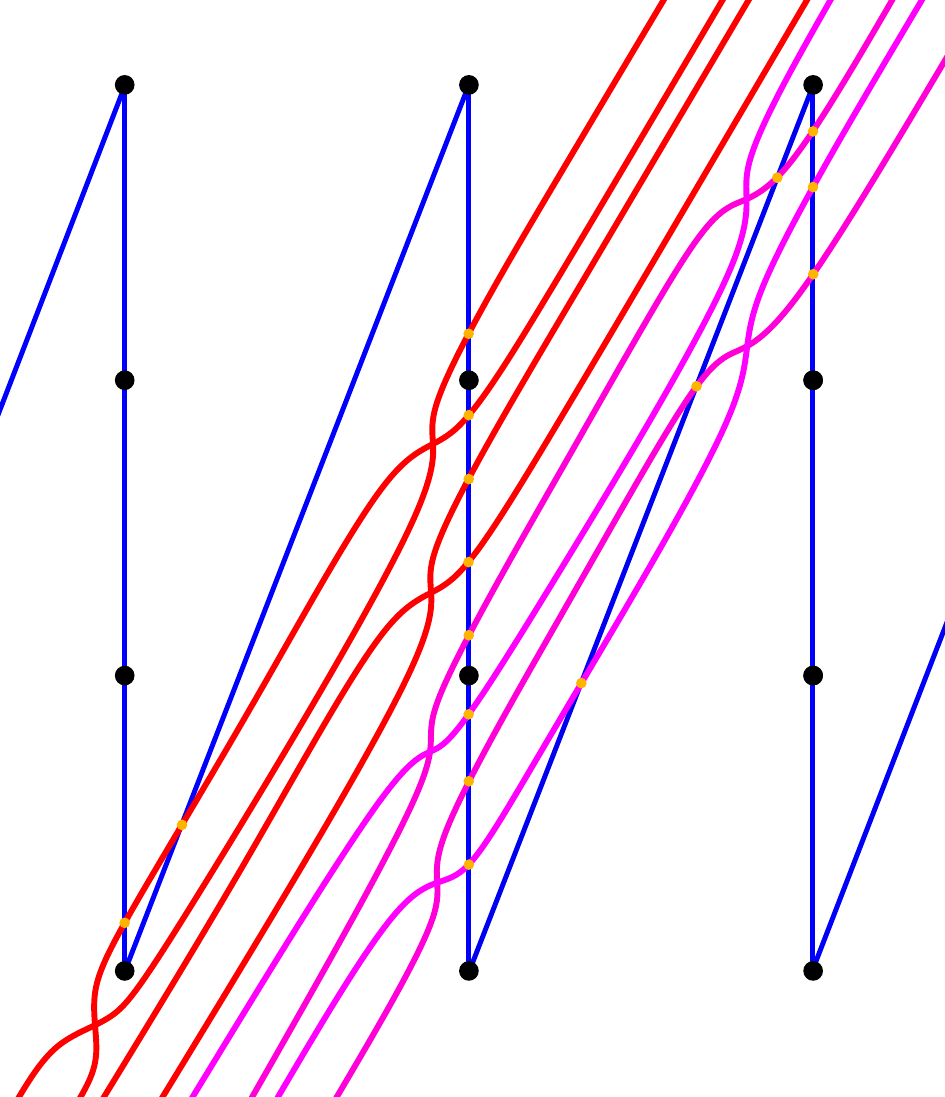
\caption{The pairing of $h(\widehat{\textit{HF}}(N))$ with $\widehat{\textit{HF}}(M)$, where $h$ is a slope 2 cyclic gluing and $J$ is an $L$-space knot with $\tau(J)=2$.}
\label{fig:plus2tauJ2Lspacepair}
\end{figure}

\section{Pairings and gradings}
\label{sec:gr}

In this section, we obstruct $(S^3 \setminus \nu T(2,3)) \cup_h N$ from being realized as $S^3_8(K)$ for $g(K)=2$ using the Maslov grading structure on $\widehat{\textit{HF}}(X)$. The relative $\Z$-grading on Heegaard Floer homology can be lifted to an absolute $\Q$-grading, and for an L-space surgery the grading of the generator of $\widehat{\textit{HF}}(S^3_p(K), [s])$ is given by $d(S^3_p(K), [s])$. In general these are easy to determine from $\textit{CFK}^{\infty}(K)$, which we did back in Lemma \ref{lem:MLemFinite} for $8$-surgery along a knot $K$ with $\widehat{\textit{HFK}}(K) \cong \widehat{\textit{HFK}}(T(2,5)$. 

While the immersed curves framework has been invaluable for ruling out most suitable knot complements $S^3 \setminus \nu J$, it is limited in its capacity to compare grading information across $\text{spin}^{c}$ structures. We will compute the relative $\Q$-grading on $\widehat{\textit{HF}}(X)$ using train tracks and their associated prototype pairing theorem (Subsection \ref{subsec:IC}), and then generate the desired obstruction by comparing grading differences. This is possible through the main result of \cite{LOT18a}, which we will give after a brief overview of the (refined) grading on bordered invariants (for manifolds with torus boundary).

From \cite[Section 11.1]{LOT18b}, the refined grading on the algebra $\mathcal{A}(\mathbb{T})$ takes values in a non-commutative group $G$ (arising as a $\Z$-central extension of $H_1(\mathbb{T})$). The group $G$ is generated by triples $(j; p, q)$ with $j, p, q \in \frac{1}{2}\Z$ and $p+q \in \Z$, and has a central element $\lambda = (1; 0, 0)$ (not to be confused with the homological longitude of $S^3 \setminus \nu J$). We will refer to $j$ as the Maslov component, and $(p,q)$ as the $\text{Spin}^{c}$ component. The group law is given by
\begin{align*}
(j_1; p_1, q_1) \cdot (j_2; p_2, q_2) = \Bigg(j_1 + j_2 
+ \left| \begin{matrix}
p_1 & q_1 \\
p_2 & q_2 \\
\end{matrix} \right|; p_1 + p_2, q_1 + q_2 \Bigg),
\end{align*}

\noindent and the gradings of the algebra elements are generated from the non-zero products on $\mathcal{A}(\mathbb{T})$ from:
\begin{align*}
\text{gr}(\rho_1) &= (-\tfrac{1}{2}; \tfrac{1}{2}, -\tfrac{1}{2}) \\
\text{gr}(\rho_2) &= (-\tfrac{1}{2}; \tfrac{1}{2}, \tfrac{1}{2}) \\
\text{gr}(\rho_3) &= (-\tfrac{1}{2}; -\tfrac{1}{2}, \tfrac{1}{2}) 
\end{align*}

Let $M_1$ be a bordered manifold. Given $\s_1 \in \text{Spin}^{c}(M_1)$, fix a base generator $\mathbf{x}_0 \in \widehat{\textit{CFA}}(M_1, \s_1)$. The module $\widehat{\textit{CFA}}(M_1, \s_1)$ is graded by the right $G$-set $G_A(M_1, \s_1) := P(\mathbf{x}_0) \backslash G$, with subgroup $P(\mathbf{x_0})$ defined in terms of periodic domains $B \in \pi_2(\mathbf{x}_0, \mathbf{x}_0)$ in a bordered Heegaard diagram for $M$. While this construction depends on the choice of base generator $\mathbf{x}_0$, different choices give isomorphic grading sets \cite[Section 10.3]{LOT18b}. We have $\pi_2(\mathbf{x}_0, \mathbf{x}_0) \cong H_2(M_1) \oplus \Z$ when $M_1$ has torus boundary, and so $P(\mathbf{x}_0)$ is cyclic if $M$ is a rational homology solid torus.

Since the bordered manifolds in this paper are rational homology solid tori, we have $\widehat{\textit{CFA}}(M_1, \s_1)$ graded by $\langle f \rangle \backslash G$ for some $f \in G$. Similarly, $\widehat{\textit{CFD}}(M_2, \s_2)$ is graded by the left $G$-set $G_D(M_2, \s) := G / \langle h \rangle$ for some $h \in G$. The tensor product $\widehat{\textit{CFA}}(M_1, \s_1) \boxtimes \widehat{\textit{CFD}}(M_2, \s_2)$ is graded by $G_A(M_1, \s_1) \times_G G_D(M_2, \s_2)$, and the grading of $\mathbf{x}_1 \otimes \mathbf{x}_2$ is $\text{gr}(\mathbf{x}_1 \otimes \mathbf{x}_2) = (\text{gr}(\mathbf{x}_1), \text{gr}(\mathbf{x}_2))$.

We can also work with rational periodic domains by extending $G$ and its multiplication over $\Q$ to obtain $G_{\Q}$, and define $G_{A, \Q}(M, \s) := \langle t \cdot f \rangle \backslash G_{\Q}$ and $G_{D, \Q}(M, \s) := G_{\Q} / \langle t \cdot h \rangle$ for $t \in \Q$. The main results of \cite{LOT18a} are combined in Theorem \ref{thm:relQ} below (in the case of manifolds with torus boundary). It states that the relative $\Q$-grading by $G_{A, \Q}(M_1, \s_1) \times_{G_{\Q}} G_{D, \Q}(M_2, \s_2)$ recovers the relative $\Q$-grading on $\widehat{\textit{CF}}(M_1 \cup_h M_2)$.

\begin{theorem}[{\cite[Theorem 1, Corollary 3.2, Remark 3.3]{LOT18a}}]
Consider the pairing $X = M_1 \cup_h M_2$, where the $M_i$ are compact, oriented 3-manifolds with torus boundary and $h$ is an orientation-reversing homeomorphism of boundaries. Suppose that $\mathbf{x}, \mathbf{y} \in \widehat{\textit{CFA}}(M_1, \alpha_1, \beta_1, \s_1) \boxtimes \widehat{\textit{CFD}}(M_2, h^{-1}(\beta_1), h^{-1}(\alpha_1), \s_2)$ are such that $\s(\mathbf{x})$ and $\s(\mathbf{y})$ are torsion and $\s(\mathbf{x})|_{M_i} = \s(\mathbf{y})|_{M_i} =: \s_i$ for $i=1,2$. Then $\text{gr}_{\Q}(\mathbf{x})$ and $\text{gr}_{\Q}(\mathbf{y})$ lie in the same $\Q$-orbit of $G_{A, \Q}(M_1, \s_1) \times_{G_{\Q}} G_{D, \Q}(M_2, \s_2)$. In particular, the $G$-set grading $\text{gr}_{\Q}$ determines the relative $\Q$-grading on $\widehat{\textit{HF}}$.
\label{thm:relQ}
\end{theorem}

Since $\widehat{\textit{CFA}}(M)$ is an $\mathcal{A}_{\infty}$-module graded by a right $G$-set, homogeneous elements $\mathbf{x} \in \widehat{\textit{CFD}}(M)$ and $\rho_{I_n} \in \mathcal{A}$ satisfy
\begin{align*}
\text{gr}(m_{k+1}(\mathbf{x}, \rho_{I_1}, \dots, \rho_{I_k})) &= \gamma^{k-1}\text{gr}(\mathbf{x})\text{gr}(\rho_{I_1})\cdots \text{gr}(\rho_{I_k}) & & \textit{if} \,\, \mathbf{x} \otimes \rho_{I_1} \otimes \cdots \otimes \rho_{I_k} \neq 0. \\ 
\end{align*}
\vspace{-1.8\intextsep}

Similarly, since $\widehat{\textit{CFD}}(M)$ is a left differential $\mathcal{A}$-module graded by a left $G$-set, homogeneous elements $\mathbf{x} \in \widehat{\textit{CFD}}(M)$ and $\rho_I \in \mathcal{A}$ satisfy both
\begin{align*}
\text{gr}(\rho_I \otimes \mathbf{x}) &= \text{gr}(\rho_I)\text{gr}(\mathbf{x}) & &\textit{if} \,\, \rho_I \otimes \mathbf{x} \neq 0 \\ 
\text{gr}(\partial \mathbf{x}) &= \gamma^{-1} \mathbf{x} & &\textit{if} \,\, \partial \mathbf{x} \neq 0.
\end{align*}

We can now use these properties to compute the refined gradings on the relevant bordered invariants, and use Theorem \ref{thm:relQ} to establish the final lemma to prove Theorem \ref{thm:main}. For the rest of the paper, let $M = S^3 \setminus \nu T(2,3)$. For our prototypical gluing $h$, we have

\begin{align*}
\widehat{\textit{HF}}(X) &= H_{\ast}(\widehat{CFA}(M, \mu, \lambda) \boxtimes \widehat{\textit{CFD}}(N, h^{-1}(\lambda), h^{-1}(\mu))) \\
&= H_{\ast}(\widehat{CFA}(M, \mu, \lambda) \boxtimes \widehat{\textit{CFD}}(N, 2\phi_1 + \phi_0, -\phi_1)).
\end{align*}

We first need to determine the refined gradings for $\widehat{\textit{CFA}}(M, \mu, \lambda)$, and both $\widehat{\textit{CFD}}(N, 2\phi_1 + \phi_0, -\phi_1, \s_1)$ and $\widehat{\textit{CFD}}(N, 2\phi_1 + \phi_0, -\phi_1, \s_0)$. We begin with the former, and work towards the latter two structures on $N$.

\subsection{Refined gradings for $\widehat{\textit{CFA}}(M, \mu, \lambda)$}
\label{subsec:M}

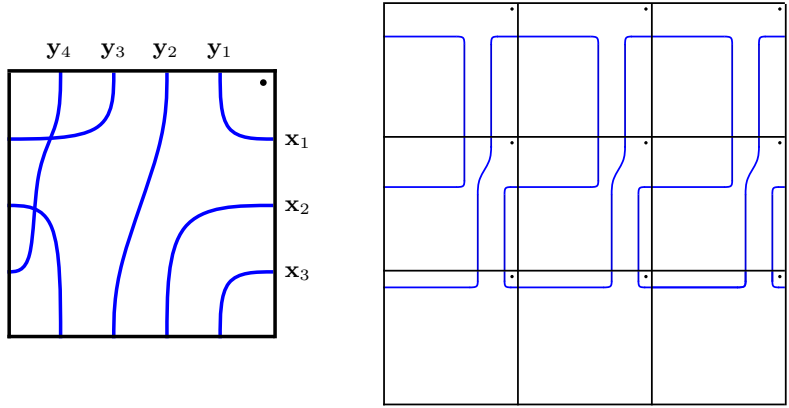
\begin{figure}[!ht]
\labellist
\small\hair 2pt
\pinlabel $\mathbf{x}_1$ at 123 112
\pinlabel $\mathbf{x}_2$ at 123 84
\pinlabel $\mathbf{x}_3$ at 123 56
\pinlabel $\mathbf{y}_1$ at 90 150
\pinlabel $\mathbf{y}_2$ at 67 150
\pinlabel $\mathbf{y}_3$ at 45 150
\pinlabel $\mathbf{y}_4$ at 22 150
\endlabellist
\centering
\includegraphics[scale=0.9]{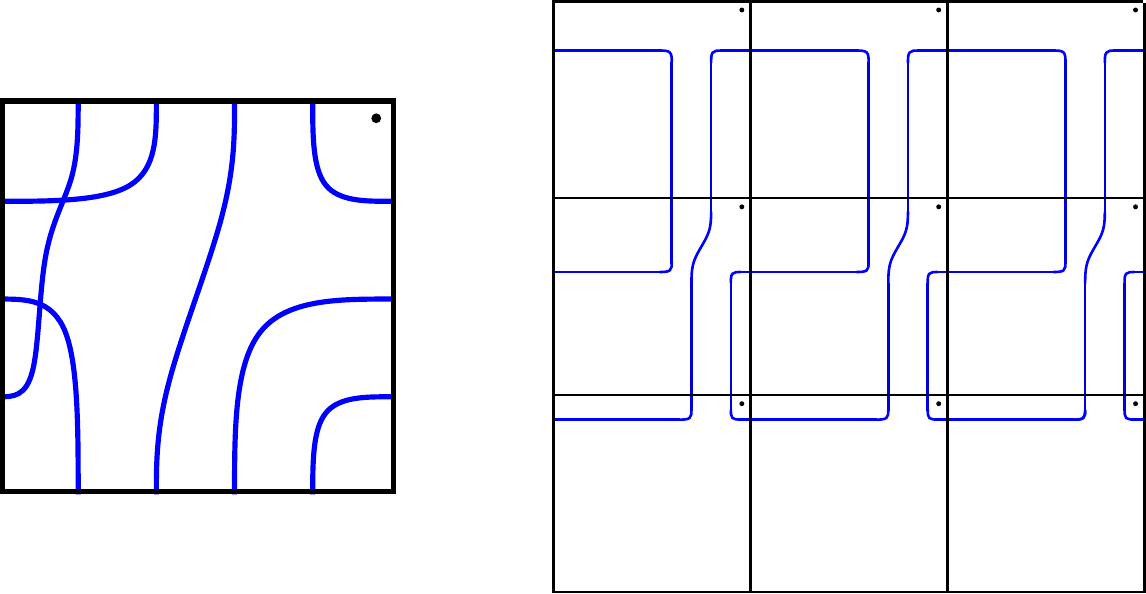}
\caption{Right: A particular representative of $\widehat{\textit{HF}}(M, \s)$ in $\widetilde{T}$. Left: The type A realization $A(\theta_M)$}
\label{fig:T23Abase3by3}
\end{figure}

\begin{wrapfigure}{r}{0.4\linewidth}
\vspace{1\intextsep}
\labellist
\small\hair 2pt
\pinlabel $\mathbf{x}_1$ at 230 99
\pinlabel $123$ at 225 148
\pinlabel $\mathbf{y}_1$ at 187 193
\pinlabel $1$ at 132 213
\pinlabel $\mathbf{x}_3$ at 85 218
\pinlabel $3$ at 35 200
\pinlabel $\mathbf{y}_4$ at -10 155
\pinlabel $2$ at -15 105
\pinlabel $\mathbf{x}_2$ at -10 45
\pinlabel $1$ at 30 7
\pinlabel $\mathbf{y}_2$ at 87 -15
\pinlabel $23$ at 145 -5
\pinlabel $\mathbf{y}_3$ at 189 9
\pinlabel $3$ at 218 55
\endlabellist
\centering
\includegraphics[scale=0.5]{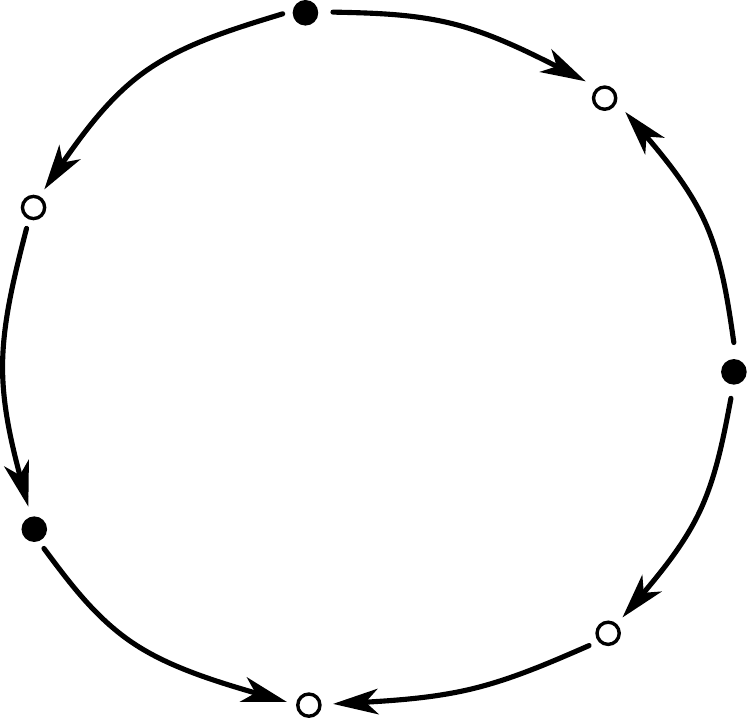}
\vspace{1\intextsep}
\caption{The decorated graph representation of $\widehat{\textit{CFD}}(M, \mu, \lambda)$}
\label{fig:MDstr}
\end{wrapfigure}

We will use the decorated graph representation of $\widehat{\textit{CFA}}(M, \mu, \lambda)$ to determine the refined gradings of its generators. To obtain this, let us reverse-engineer $A(\theta_M)$ from $\widehat{\textit{HF}}(M, \s)$ by projecting an appropriate representative onto $T_M$ (with $z$ at $(1-\epsilon, 1-\epsilon)$). The form of $\widehat{\textit{HF}}(M, \s)$ in $\widetilde{T}$ is simple due to Lemma \ref{lem:lspaceJ}. Seen in $\overline{T}_M$, the curve invariant initially wraps around a lift of $z$ at height $\tfrac{1}{2}$, then around a lift of $z$ at height $-\tfrac{1}{2}$ before wrapping back around the cylinder since $\tau(T(2,3))=1$ and $\epsilon(T(2,3))=1$. The choice representative in $\widetilde{T}$ is shown in Figure \ref{fig:T23Abase3by3}, together with the projection onto $T_M$ yielding the desired $A(\theta_M)$. Recall that the decorated graph representing $\widehat{\textit{CFD}}(M, \mu, \lambda)$ can be extracted from $A(\theta_M)$ using the edge identifications from Figure \ref{fig:ReebChords}. The result is shown in Figure \ref{fig:MDstr}. \\

\begin{wrapfigure}{r}{0.4\linewidth}
\vspace{1\intextsep}
\labellist
\small\hair 2pt
\pinlabel $\mathbf{x}_1$ at 230 99
\pinlabel $321$ at 225 148
\pinlabel $\mathbf{y}_1$ at 187 193
\pinlabel $3$ at 132 213
\pinlabel $\mathbf{x}_3$ at 85 218
\pinlabel $1$ at 35 200
\pinlabel $\mathbf{y}_4$ at -10 155
\pinlabel $2$ at -15 105
\pinlabel $\mathbf{x}_2$ at -10 45
\pinlabel $3$ at 30 7
\pinlabel $\mathbf{y}_2$ at 87 -15
\pinlabel $21$ at 145 -5
\pinlabel $\mathbf{y}_3$ at 189 9
\pinlabel $1$ at 218 55
\endlabellist
\centering
\includegraphics[scale=0.5]{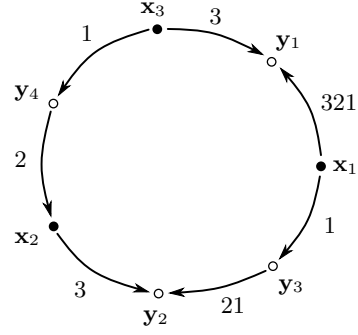}
\vspace{1\intextsep}
\caption{The decorated graph representation of $\widehat{\textit{CFA}}(M, \mu, \lambda)$}
\label{fig:MAstr}
\end{wrapfigure}

Finally, an algorithm of Hedden and Levine may be followed to obtain the decorated graph representation of $\widehat{\textit{CFA}}(M, \mu, \lambda)$ \cite{HL16}. The idempotent splitting according to vertex labels remains the same, but the edge labels are interpreted differently. First, rewrite them according to the bijection $1 \leftrightarrow 3$. Given a directed path from $x$ to $y$, construct a sequence $I = I_1, \dots I_k$ and assign the multiplication $m_{k+1}(x \otimes \rho_{I_1} \otimes \cdots \otimes \rho_{I_k}) = y$. Reading the edge labels of the directed path in order, form $I$ by regrouping to find the minimum $k$ so that each $I_j$ is an element of $\{ 1, 2, 3, 12, 13, 23, 123 \}$. For example, the edge labeled $\{23\}$ from $\mathbf{y}_3$ to $\mathbf{y}_2$ in Figure \ref{fig:MDstr} gives the sequence $I = \{2, 1\}$ and the product $m_3(\mathbf{y}_3, \rho_2, \rho_1) = \mathbf{y}_2$. The resulting decorated graph is shown in Figure \ref{fig:MAstr}. 

Reading the edge labels gives each generator's contribution to the $m_k$, and allows us to determine their refined gradings taking values in $\langle f \rangle \backslash G$, up to some indeterminacy $f$. Set $\mathbf{x}_1$ to be the base generator. Since the decorated graph for $\widehat{\textit{CFA}}(M, \mu, \lambda)$ is valence 2 and connected, we can traverse an oriented path from $\mathbf{x}_1$ to itself through the other generators to simultaneously compute the refined gradings of the remaining generators and the indeterminacy $f$.

The relevant contributions to the $m_k$ are the following:
\begin{align*}
&1) \,\, m_4(\mathbf{x}_1, \rho_3, \rho_2, \rho_1) = \mathbf{y}_1 & & 4) \,\, m_3(\mathbf{y}_3, \rho_2, \rho_1) = \mathbf{y}_2 \\
&2) \,\, m_2(\mathbf{x}_3, \rho_3) = \mathbf{y}_1 & & 5) \,\, m_2(\mathbf{x}_1, \rho_1) = \mathbf{y}_3 \\
&3) \,\, m_2(\mathbf{x}_3, \rho_{123}) = \mathbf{y}_2 & &
\end{align*}

For the first term $m_4(\mathbf{x}_1, \rho_3, \rho_2, \rho_1) = \mathbf{y}_1$, we see that

\begin{align*}
\text{gr}(\mathbf{y}_1) &= \gamma^2 \text{gr}(\mathbf{x}_1)\text{gr}(\rho_3)\text{gr}(\rho_2)\text{gr}(\rho_1) \\
&= \text{gr}(\mathbf{x}_1) \gamma^2 \text{gr}(\rho_3)\left(-\tfrac{1}{2}; \tfrac{1}{2}, \tfrac{1}{2} \right) \left(-\tfrac{1}{2}; \tfrac{1}{2}, -\tfrac{1}{2} \right) \\
&= \text{gr}(\mathbf{x}_1) \gamma^2 \left(-\tfrac{1}{2}; -\tfrac{1}{2}, \tfrac{1}{2} \right) \left(-\tfrac{3}{2}; 1, 0\right) \\
&= \text{gr}(\mathbf{x}_1) \left(2; 0, 0 \right) \left(-\tfrac{5}{2}; \tfrac{1}{2}, \tfrac{1}{2} \right) \\
&= \langle f \rangle \backslash \left(-\tfrac{1}{2}; \tfrac{1}{2}, \tfrac{1}{2} \right)
\end{align*}

Performing this for the remaining generators, we obtain their (undetermined) refined gradings in $\langle f \rangle \backslash G$:

\begin{align*}
\text{gr}(\mathbf{x}_1) &= \langle f \rangle \backslash \left(0; 0, 0 \right) & & \text{gr}(\mathbf{y}_1) = \langle f \rangle \backslash \left(-\tfrac{1}{2}; \tfrac{1}{2}, \tfrac{1}{2}\right) \\
\text{gr}(\mathbf{x}_2) &= \langle f \rangle \backslash \left(-1; 2, 0 \right) & & \text{gr}(\mathbf{y}_2) = \langle f \rangle \backslash \left(-\tfrac{1}{2}; \tfrac{3}{2}, \tfrac{1}{2} \right) \\
\text{gr}(\mathbf{x}_3) &= \langle f \rangle \backslash \left(-\tfrac{1}{2}; 1, 0 \right) & & \text{gr}(\mathbf{y}_3) = \langle f \rangle \backslash \left(\tfrac{1}{2}; \tfrac{1}{2}, \tfrac{1}{2} \right) \\
& & & \text{gr}(\mathbf{y}_4) = \langle f \rangle \backslash \left(-\tfrac{3}{2}; \tfrac{3}{2}, -\tfrac{1}{2} \right)
\end{align*}

The final term $m_2(\mathbf{x}_1, \rho_1) = \mathbf{y}_3$ allows us to pin down $f$:

\begin{align*}
\text{gr}(\mathbf{x}_1) \text{gr}(\rho_1) &= \text{gr}(\mathbf{y}_3) \\
\Rightarrow \text{gr}(\mathbf{x}_1) &= \text{gr}(\mathbf{y}_3)\text{gr}(\rho_1)^{-1} \\
&= \langle f \rangle \backslash \left(\tfrac{1}{2}; \tfrac{1}{2}, \tfrac{1}{2} \right) \left(\tfrac{1}{2}; -\tfrac{1}{2}, \tfrac{1}{2} \right) \\
&= \langle f \rangle \backslash \left(\tfrac{3}{2}; 0, 1 \right)
\end{align*}

Thus, $f = \left(\tfrac{3}{2}; 0, 1 \right)$ since it is primitive. The refined gradings for generators of $\widehat{\textit{CFA}}(M, \mu, \lambda)$ are the following:

\begin{align*}
\text{gr}(\mathbf{x}_1) &= \langle (\tfrac{3}{2}; 0, 1) \rangle \backslash \left(0; 0, 0 \right) & & \text{gr}(\mathbf{y}_1) = \langle (\tfrac{3}{2}; 0, 1) \rangle \backslash \left(-\tfrac{1}{2}; \tfrac{1}{2}, \tfrac{1}{2}\right) \\
\text{gr}(\mathbf{x}_2) &= \langle (\tfrac{3}{2}; 0, 1) \rangle \backslash \left(-1; 2, 0 \right) & & \text{gr}(\mathbf{y}_2) = \langle (\tfrac{3}{2}; 0, 1) \rangle \backslash \left(-\tfrac{1}{2}; \tfrac{3}{2}, \tfrac{1}{2} \right) \\
\text{gr}(\mathbf{x}_3) &= \langle (\tfrac{3}{2}; 0, 1) \rangle \backslash \left(-\tfrac{1}{2}; 1, 0 \right) & & \text{gr}(\mathbf{y}_3) = \langle (\tfrac{3}{2}; 0, 1) \rangle \backslash \left(\tfrac{1}{2}; \tfrac{1}{2}, \tfrac{1}{2} \right) \\
& & & \text{gr}(\mathbf{y}_4) = \langle (\tfrac{3}{2}; 0, 1) \rangle \backslash \left(-\tfrac{3}{2}; \tfrac{3}{2}, -\tfrac{1}{2} \right)
\end{align*}

\subsection{Refined gradings for $\widehat{\textit{CFD}}(N, 2\phi_1+\phi_0, -\phi_1, \s_0)$}
\label{subsec:N0}

In order to obtain the desired decorated graph representation, we can reverse-engineer its form from $A(\theta_{h(N, \s_0)})$ (shown on the right side of Figure \ref{fig:N1and2Abase} back in Subsection \ref{subsec:IC}.) This involves performing a series of Dehn twists and reflections on $T_N$ for the type A realization of $\widehat{\textit{CFD}}(N, \phi_1, \phi_0, \s_0)$, illustrated in Figure \ref{fig:N0twists}. First, we use the reflection $r(y=\tfrac{1}{2})$ to obtain $\widehat{\textit{CFD}}(N, -\phi_1, \phi_0, \s_0)$. Second, Dehn twisting -2 times along $-\phi_1$ yields $\widehat{\textit{CFD}}(N, -\phi_1, 2\phi_1 + \phi_0, \s_0)$. Finally we use the reflection $r(y=x)$, giving the bottom-right collection of curves $A(\theta_{h(N, \s_0)})$.

\begin{figure}[!ht]
\centering
\includegraphics[scale=1]{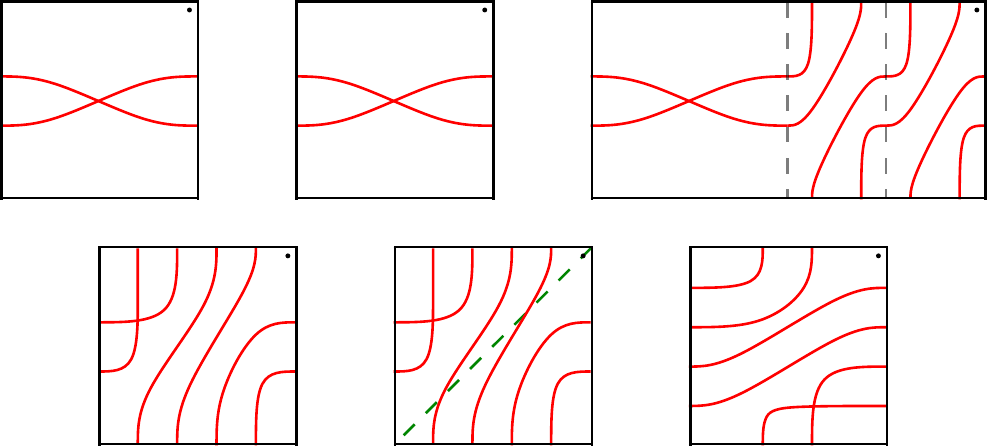}
\caption{The series of Dehn twists and reflections required to obtain the type A realization of $\widehat{\textit{CFD}}(N, 2\phi_1 + \phi_0, -\phi_1, \s_0)$. }
\label{fig:N0twists}
\end{figure}

The corresponding decorated graph representation of $\widehat{\textit{CFD}}(N, 2\phi_1 + \phi_0, -\phi_1, \s_0)$ is shown in Figure \ref{fig:N0DstrPlus2}. Label the $\iota_0$ generators by $\mathbf{a}_i$ and the $\iota_1$ generaters by $\mathbf{b}_i$. Reading the edge labels gives each generator's contribution to $\delta^1$, and allows us to determine their refined gradings in $G / \langle h_0 \rangle$ up to some indeterminacy $h$. Set $\mathbf{a}_1$ to be the base generator. Since the decorated graph at hand is valence 2 and connected, we can traverse an oriented path from $\mathbf{a}_1$ to itself through the other generators to simultaneously compute the refined gradings of the remaining generators and the indeterminacy $h_0$. This is performed below using the properties of the refined grading.

\begin{figure}[!ht]
\vspace{1\intextsep}
\labellist
\small\hair 2pt
\pinlabel $\mathbf{a}_1$ at 167 70
\pinlabel $\rho_{12}$ at 158 112
\pinlabel $\mathbf{a}_3$ at 122 145
\pinlabel $\rho_{1}$ at 77 153
\pinlabel $\mathbf{b}_1$ at 32 145
\pinlabel $\rho_{3}$ at 0 115
\pinlabel $\mathbf{a}_2$ at -15 70
\pinlabel $\rho_{12}$ at 0 25
\pinlabel $\mathbf{a}_4$ at 32 -5
\pinlabel $\rho_{1}$ at 77 -10
\pinlabel $\mathbf{b}_2$ at 122 -5
\pinlabel $\rho_{3}$ at 155 25
\endlabellist
\centering
\includegraphics[scale=0.65]{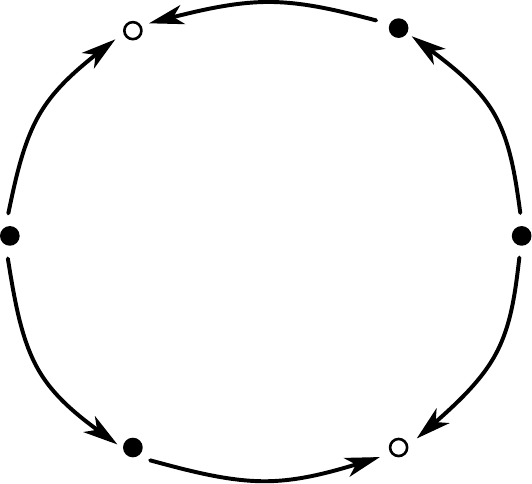}
\caption{The decorated graph representation of $\widehat{\textit{CFD}}(N, \phi_0+2\phi_1, -\phi_1, \s_0)$}
\label{fig:N0DstrPlus2}
\end{figure}
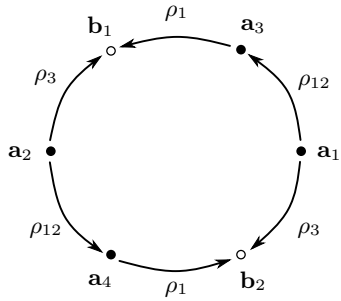

The relevant contributions to $\delta^1$ are

\begin{align*}
1) \,\, \rho_{12} \otimes \mathbf{a}_3 \in \delta^1 \mathbf{a}_1 & & 4) \,\, \rho_{12} \otimes \mathbf{a}_4 \in \delta^1 \mathbf{a}_2 \\
2) \,\, \rho_{1} \otimes \mathbf{b}_1 \in \delta^1 \mathbf{a}_3 & & 5) \,\, \rho_{1} \otimes \mathbf{b}_2 \in \delta^1 \mathbf{a}_4 \\
3) \,\, \rho_{3} \otimes \mathbf{b}_1 \in \delta^1 \mathbf{a}_2 & & 6) \,\, \rho_3 \otimes \mathbf{b}_2 \in \delta^1 \mathbf{a}_1
\end{align*}

For the first term $\rho_{12} \otimes \mathbf{a}_3 \in \delta^1 \mathbf{a}_1$, we see that

\begin{align*}
\lambda^{-1}\text{gr}(\mathbf{a}_1) &= \text{gr}(\rho_{12})\text{gr}(\mathbf{a}_3) \\
\lambda^{-1}\text{gr}(\mathbf{a}_1) &= \text{gr}(\rho_1)\text{gr}(\rho_2)\text{gr}(\mathbf{a}_3) \\
\Rightarrow \text{gr}(\mathbf{a}_3) &= \text{gr}(\rho_2)^{-1} \text{gr}(\rho_1)^{-1} \lambda^{-1} \text{gr}(\mathbf{a}_1) \\
&= \left(\tfrac{1}{2}; -\tfrac{1}{2}, -\tfrac{1}{2} \right) \left(\tfrac{1}{2}; -\tfrac{1}{2}, \tfrac{1}{2} \right) \lambda^{-1} \text{gr}(\mathbf{a}_1) \\
&= \left(\tfrac{1}{2}; -1, 0 \right) \left(-1; 0,0 \right) \text{gr}(\mathbf{a}_1) \\
&= \left(-\tfrac{1}{2}; -1, 0 \right)/\langle h_0 \rangle.
\end{align*}

Performing this for the remaining generators, we obtain their (undetermined) refined gradings in $G/\langle h_0 \rangle$.

\begin{align*}
\text{gr}(\mathbf{a}_1) &= \left(0; 0, 0 \right)/\langle h_0 \rangle & & \text{gr}(\mathbf{a}_4) = \left(-1; -3, 1 \right)/\langle h_0 \rangle \\
\text{gr}(\mathbf{a}_2) &= \left(\tfrac{1}{2}; -2, 1 \right)/\langle h_0 \rangle & & \text{gr}(\mathbf{b}_1) = \left(-\tfrac{1}{2}; -\tfrac{3}{2}, \tfrac{1}{2} \right)/\langle h_0 \rangle \\
\text{gr}(\mathbf{a}_3) &= \left(-\tfrac{1}{2}; -1, 0 \right)/\langle h_0 \rangle & & \text{gr}(\mathbf{b}_2) = \left(-\tfrac{1}{2}; -\tfrac{7}{2}, \tfrac{3}{2} \right)/\langle h_0 \rangle
\end{align*}

The final term $\rho_3 \otimes \mathbf{b}_2 \in \delta^1 \mathbf{a}_1$ allows us to pin down $h_0$:

\begin{align*}
\lambda^{-1}\text{gr}(\mathbf{a}_1) &= \text{gr}(\rho_3)\text{gr}(\mathbf{b}_2) \\
\Rightarrow \text{gr}(\mathbf{a}_1) &= \lambda \text{gr}(\rho_3)\text{gr}(\mathbf{b}_2) \\
&= \left(1; 0, 0 \right) \left(-\tfrac{1}{2}; -\tfrac{1}{2}, \tfrac{1}{2} \right) \text{gr}(\mathbf{b}_2) \\
&= \left(\tfrac{1}{2}; -\tfrac{1}{2}, \tfrac{1}{2} \right) \left(-\tfrac{1}{2}; -\tfrac{7}{2}, \tfrac{3}{2} \right) \text{gr}(\mathbf{a}_1) \\
&= \left(1; -4, 2 \right)/\langle h_0 \rangle.
\end{align*}

Then $h_0$ divides $(-1; 4, -2)$, and so the two possibilities for $h_0$ are $(-\tfrac{1}{2}; 2, -1)$ or $(-1; 4, -2)$. Supposing the former to generate a contradiction, we would have $\text{gr}(\mathbf{a}_2) = \text{gr}(\mathbf{a}_1)$ in $G/\langle h_0 \rangle$, which in turn forces $\s(\mathbf{x}_1 \boxtimes \mathbf{a_1}) = \s(\mathbf{x}_1 \boxtimes \mathbf{a_2})$. This contradicts the discussion in Subsection \ref{subsec:Jnontrivial}, where these generators correspond to separate $\text{spin}^{c}$ structures because no closed, piecewise smooth path in $\widetilde{T}$ exists to connect them. Then with $h_0 = (-1; 4, -2)$, the refined gradings for generators of $\widehat{\textit{CFD}}(N, 2\phi_1 + \phi_0, -\phi_1, \s_1)$ are the following:

\begin{align*}
\text{gr}(\mathbf{a}_1) &= \left(0; 0, 0 \right)/\langle \left(-1; 4, -2 \right) \rangle & & \text{gr}(\mathbf{a}_4) = \left(-1; -3, 1 \right)/\langle \left(-1; 4, -2 \right) \rangle \\
\text{gr}(\mathbf{a}_2) &= \left(\tfrac{1}{2}; -2, 1 \right)/\langle \left(-1; 4, -2 \right) \rangle & & \text{gr}(\mathbf{b}_1) = \left(-\tfrac{1}{2}; -\tfrac{3}{2}, \tfrac{1}{2} \right)/\langle \left( -1; 4, -2 \right) \rangle \\
\text{gr}(\mathbf{a}_3) &= \left(-\tfrac{1}{2}; -1, 0 \right)/\langle \left(-1; 4, -2 \right) \rangle & & \text{gr}(\mathbf{b}_2) = \left(-\tfrac{1}{2}; -\tfrac{7}{2}, \tfrac{3}{2} \right)/\langle \left(-1; 4, -2 \right) \rangle
\end{align*}

It is interesting to note that
\begin{align*}
\text{gr}_{\Q}(\mathbf{a}_2) &= \left(\tfrac{1}{2}; -2, 1 \right)/\langle \left( -1; 4, -2 \right) \rangle \\
&= \left(\tfrac{1}{2}; -2, 1 \right) \left( -\tfrac{1}{2}; 2, -1 \right) /\langle \left( -1; 4, -2 \right) \rangle \\
&= \left(0; 0, 0 \right) /\langle \left( -1; 4, -2 \right) \rangle \\
&= \text{gr}_{\Q}(\mathbf{a}_1),
\end{align*}
by acting over $\Q$. The same holds true for the other generators, with $\text{gr}_{\Q}(\mathbf{a}_4) = \text{gr}_{\Q}(\mathbf{a}_3)$ and $\text{gr}_{\Q}(\mathbf{b}_1) = \text{gr}_{\Q}(\mathbf{b}_2)$. Generators of $HF(\widehat{\textit{HF}}(S^3 \setminus \nu J), h(\widehat{\textit{HF}}(N, \s_0)))$ come in pairs exhibiting this feature, which suggests that they are associated to conjugate $\text{spin}^{c}$ structures on $X$. That this behavior does not depend on $J$ is enough to show this, but we will not need it for our purposes.

Finally, we will need the corresponding type D realization $D(\theta_{h(N, \s_0)})$ for pairing in Subsection \ref{subsec:grcomputation}, presented now in Figure \ref{fig:N0DfromA}.

\begin{figure}[!ht]
\vspace{1\intextsep}
\labellist
\small\hair 2pt
\pinlabel $\mathbf{a}_1$ at 148 108
\pinlabel $\mathbf{a}_2$ at 148 80
\pinlabel $\mathbf{a}_3$ at 148 53
\pinlabel $\mathbf{a}_4$ at 148 25
\pinlabel $\mathbf{b}_1$ at 86 147
\pinlabel $\mathbf{b}_2$ at 50 147
\pinlabel $\mathbf{a}_1$ at 232 -9
\pinlabel $\mathbf{a}_2$ at 259 -9
\pinlabel $\mathbf{a}_3$ at 286 -9
\pinlabel $\mathbf{a}_4$ at 313 -9
\pinlabel $\mathbf{b}_1$ at 193 84
\pinlabel $\mathbf{b}_2$ at 193 50
\endlabellist
\centering
\includegraphics[scale=0.8]{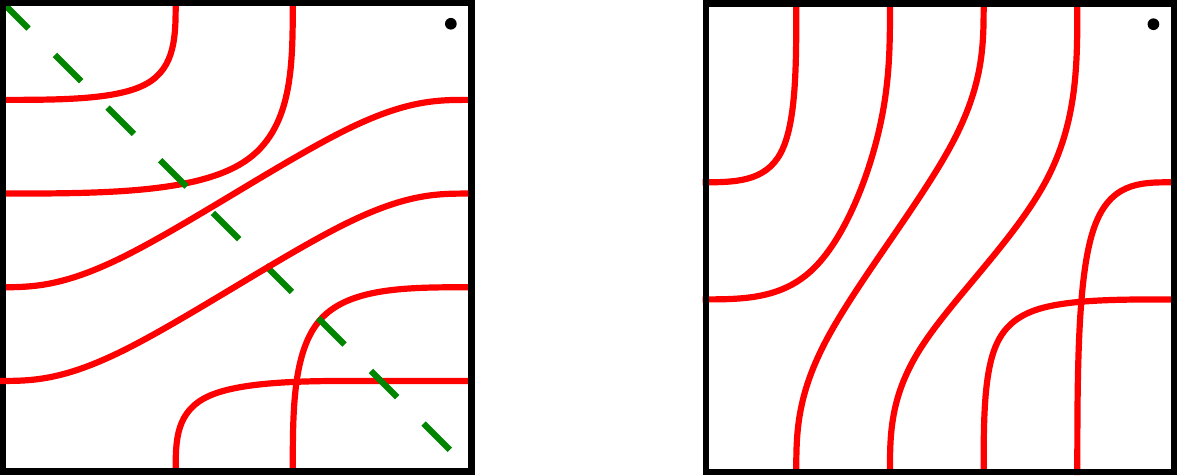}
\vspace{1\intextsep}
\caption{Left: $A(\theta_{h(N, \s_0)})$. Right: Obtaining $D(\theta_{h(N, \s_0)})$ via $r(y=-x)$.}
\label{fig:N0DfromA}
\end{figure}
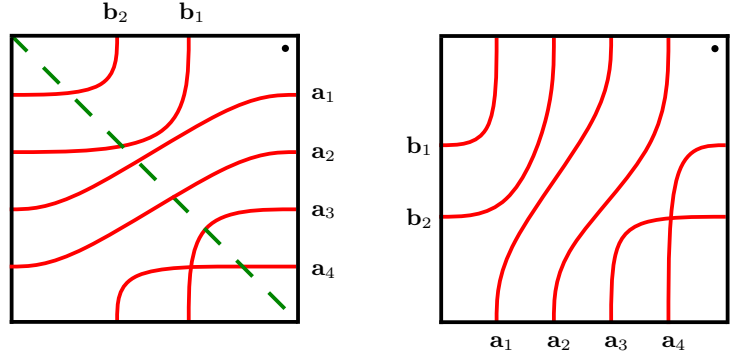

\subsection{Refined gradings for $\widehat{\textit{CFD}}(N, 2\phi_1+\phi_0, -\phi_1, \s_1)$}
\label{subsec:N1}

Recall from Subsection \ref{subsec:IC} that the decorated graph on the left side of Figure \ref{fig:N0and1DstrNoTwist} gives rise to $A(\theta_{N, \s_1})$ shown in Figure \ref{fig:N1and2Abase}. The same series of Dehn twists and reflections as those in the previous subsection can be applied to obtain $A(\theta_{h(N, \s_1)})$ and $D(\theta_{h(N, \s_1)})$. The end results are depicted in Figure \ref{fig:N1DfromA}.

\begin{figure}[!ht]
\vspace{1\intextsep}
\labellist
\small\hair 2pt
\pinlabel $\mathbf{z}_1$ at 148 108
\pinlabel $\mathbf{z}_2$ at 148 80
\pinlabel $\mathbf{z}_3$ at 148 53
\pinlabel $\mathbf{z}_4$ at 148 25
\pinlabel $\mathbf{y}_1$ at 86 147
\pinlabel $\mathbf{y}_2$ at 50 147
\pinlabel $\mathbf{z}_1$ at 232 -9
\pinlabel $\mathbf{z}_2$ at 259 -9
\pinlabel $\mathbf{z}_3$ at 286 -9
\pinlabel $\mathbf{z}_4$ at 313 -9
\pinlabel $\mathbf{y}_1$ at 193 84
\pinlabel $\mathbf{y}_2$ at 193 50
\endlabellist
\centering
\includegraphics[scale=0.8]{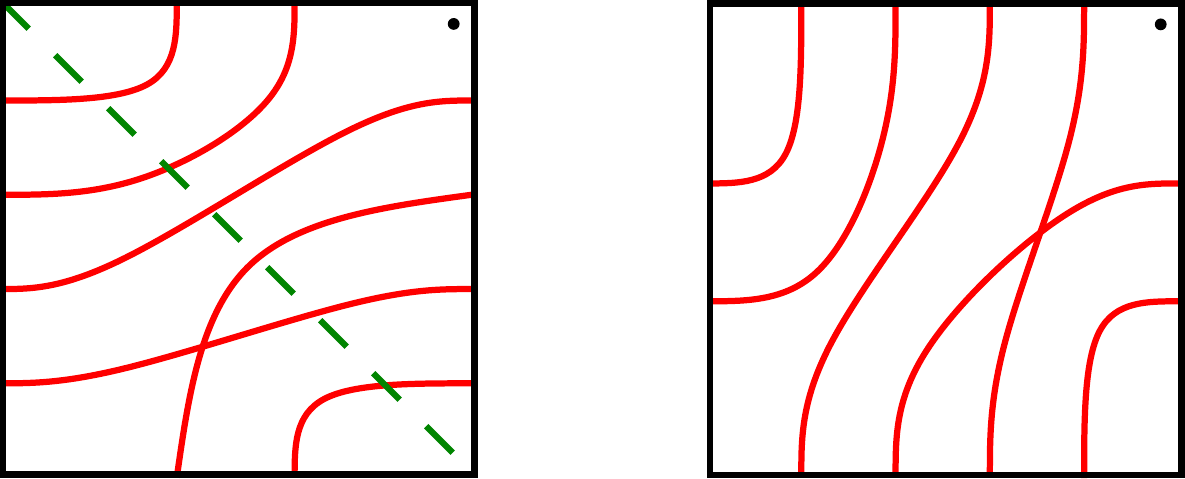}
\vspace{1\intextsep}
\caption{Left: $A(\theta_{h(N, \s_1)})$. Right: Obtaining $D(\theta_{h(N, \s_1)})$ via $r(y=-x)$.}
\label{fig:N1DfromA}
\end{figure}
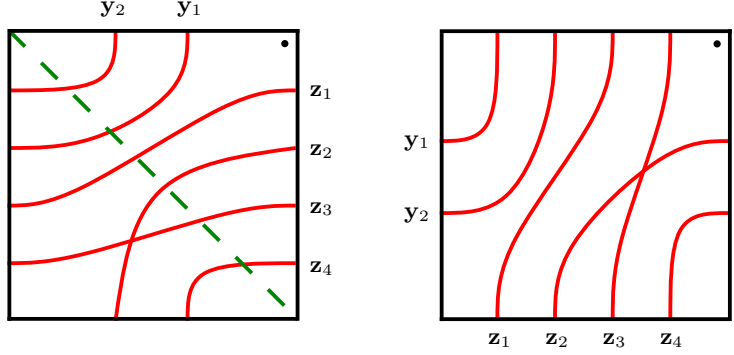

Label the $\iota_0$ generators by $\mathbf{z}_i$ and the $\iota_1$ generators by $\mathbf{w}_i$. The decorated graph for $\widehat{\textit{CFD}}(N, 2\phi_1 + \phi_0, -\phi_1, \s_0)$ is then given by Figure \ref{fig:N1DstrPlus2}.

\begin{figure}[!ht]
\vspace{1\intextsep}
\labellist
\small\hair 2pt
\pinlabel $\mathbf{z}_1$ at 167 70
\pinlabel $\rho_{12}$ at 158 112
\pinlabel $\mathbf{z}_3$ at 122 145
\pinlabel $\rho_{12}$ at 77 153
\pinlabel $\mathbf{z}_4$ at 32 145
\pinlabel $\rho_{1}$ at 0 115
\pinlabel $\mathbf{w}_1$ at -15 70
\pinlabel $\rho_{3}$ at 0 25
\pinlabel $\mathbf{z}_2$ at 32 -5
\pinlabel $\rho_{1}$ at 77 -10
\pinlabel $\mathbf{w}_2$ at 122 -5
\pinlabel $\rho_{3}$ at 155 25
\endlabellist
\centering
\includegraphics[scale=0.65]{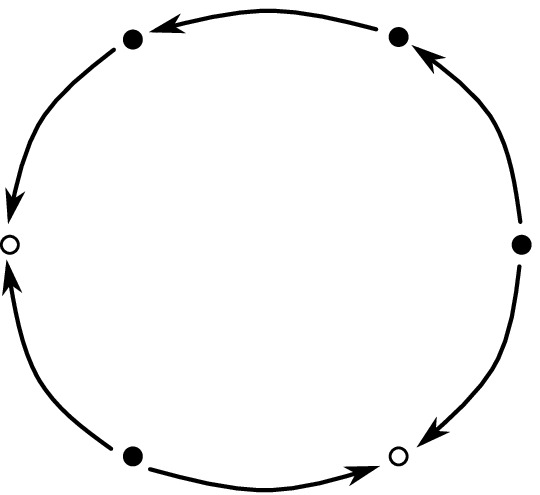}
\caption{The decorated graph representation of $\widehat{\textit{CFD}}(N, \phi_0+2\phi_1, -\phi_1, \s_0)$}
\label{fig:N1DstrPlus2}
\end{figure}

Using similar computations to those in the previous subsection, we have that the refined gradings for generators of $\widehat{\textit{CFD}}(N, 2\phi_1 + \phi_0, -\phi_1, \s_1)$ lie in $G/\langle \left(-3; 4, -2 \right) \rangle$ and are the following:

\begin{align*}
\text{gr}(\mathbf{z}_1) &= \left(0; 0, 0 \right)/\langle \left(-3; 4, -2 \right) \rangle & & \text{gr}(\mathbf{z}_4) = \left(-1; -2, 0 \right)/\langle \left(-3; 4, -2 \right) \rangle \\
\text{gr}(\mathbf{z}_2) &= \left(1; -3, 1 \right)/\langle \left(-3; 4, -2 \right) \rangle & & \text{gr}(\mathbf{w}_1) = \left(-\tfrac{1}{2}; -\tfrac{5}{2}, \tfrac{1}{2} \right)/\langle \left(-3; 4, -2 \right) \rangle \\
\text{gr}(\mathbf{z}_3) &= \left(-\tfrac{1}{2}; -1, 0 \right)/\langle \left(-3; 4, -2 \right) \rangle & & \text{gr}(\mathbf{w}_2) = \left(\tfrac{3}{2}; -\tfrac{7}{2}, \tfrac{3}{2} \right)/\langle \left(-3; 4, -2 \right) \rangle
\end{align*}

\subsection{Refined gradings for generators of $H_{\ast}(\widehat{\textit{CFA}}(M, \mu, \lambda) \boxtimes \widehat{\textit{CFD}}(N, 2\phi_1 + \phi_0, -\phi_1))$}
\label{subsec:grcomputation}

With $A(\theta_M)$ from Subsection \ref{subsec:M}, and both $D(\theta_{h(N, \s_0)})$ from Subsection \ref{subsec:N0} and $D(\theta_{h(N, \s_1)})$ from Subsection \ref{subsec:N1}, we can now compute the relative Maslov grading differences for the generators of $\widehat{\textit{HF}}(X)$. We do this first in detail for those $\mathfrak{t} \in \pi^{-1}(\s \times \s_0)$, and then briefly perform the same methods for those $\mathfrak{t} \in \pi^{-1}(\s \times \s_1)$ for completeness.

\begin{figure}[!ht]
\centering
\includegraphics[scale=0.7]{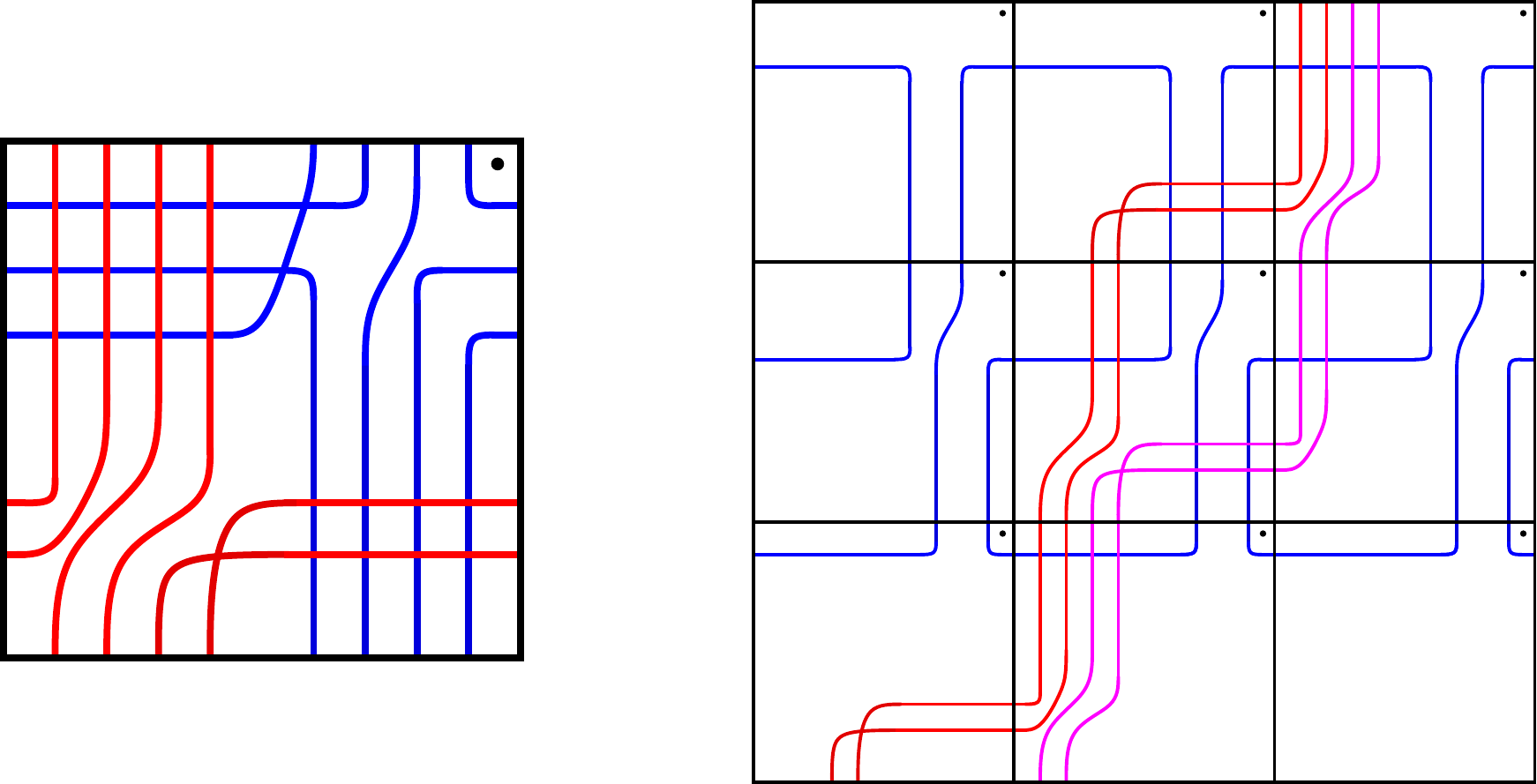}
\caption{Including $A(\theta_M)$ and $D(\theta_{h(N, \s_0)})$ in $T_M$, and lifting intersections to $\widetilde{T}$.}
\label{fig:T23N0BoxLift}
\end{figure}

Include $A(\theta_M)$ into the first quadrant of $T_M$ and $D(\theta_{h(N, \s_1)})$ into the third quadrant of $T_M$, extending both horizontally and vertically. This results in the configuration shown on the left of Figure \ref{fig:T23N0BoxLift}. According to the train track version of the pairing theorem, the summands $\widehat{\textit{HF}}(X, \mathfrak{t})$ for which $\mathfrak{t} \in \pi^{-1}(\s \times \s_0)$ are given by $H_{\ast}(\mathcal{C}(A(\theta_M), D(\theta_{h(N, \s_0)}), d^{\theta}))$. It is significantly more convenient to lift to $\widetilde{T}$, such as on the right of Figure \ref{fig:T23N0BoxLift}, to see which generators are annihilated under Floer homology.

\begin{figure}[!ht]
\centering
\def\svgscale{0.7}
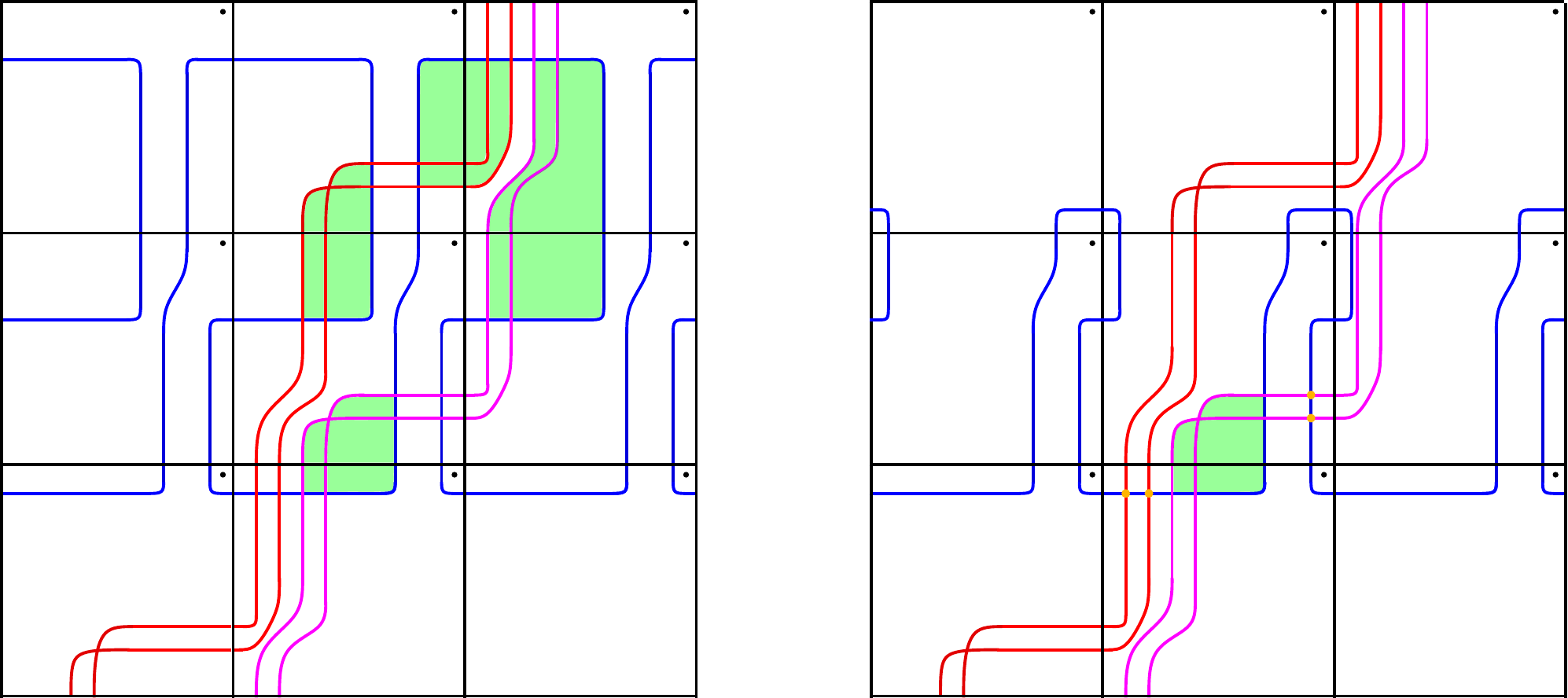
\caption{Left: Bigons between lifted intersections of $A(\theta_M)$ and $D(\theta_{h(N, \s_0)})$ to $\widetilde{T}$. Right: Homotoped $A(\theta_M)$, removing most generators annihilated in intersection Floer homology and leaving lifts of four generators that survive.}
\label{fig:T23N0DisksToped}
\end{figure}

This is done in stages throughout Figures \ref{fig:T23N0BoxLift} and \ref{fig:T23N0DisksToped} to illustrate the homotopy performed to remove most bigons that do not cover a basepoint. Ultimately, the four generators are in correspondence with the four surviving intersections in Figure \ref{fig:LargeT23boxN0}. They are $\mathbf{x}_1 \boxtimes \mathbf{a}_1$, $\mathbf{x_1} \boxtimes \mathbf{a}_2$, $\mathbf{y}_1 \boxtimes \mathbf{b}_1$, and $\mathbf{y_1} \boxtimes \mathbf{b}_2$. Their refined gradings lie in $G_{A, \Q}(M, \s) \times_{G_{\Q}} G_{D, \Q}(h(N, \s_0))$, which may be determined using the information from Subsections \ref{subsec:M} and \ref{subsec:N1} together with Theorem \ref{thm:relQ}.

\begin{figure}[!ht]
\vspace{1\intextsep}
\labellist
\small\hair 2pt
\pinlabel $\mathbf{x}_1$ at -10 197
\pinlabel $\mathbf{x}_2$ at -10 169
\pinlabel $\mathbf{x}_3$ at -10 141
\pinlabel $\mathbf{a}_1$ at 25 236
\pinlabel $\mathbf{a}_2$ at 48 236
\pinlabel $\mathbf{a}_3$ at 70 236
\pinlabel $\mathbf{a}_4$ at 93 236
\pinlabel $\mathbf{b}_2$ at 240 68
\pinlabel $\mathbf{b}_1$ at 240 45
\pinlabel $\mathbf{y}_4$ at 136 -10
\pinlabel $\mathbf{y}_3$ at 158 -10
\pinlabel $\mathbf{y}_2$ at 181 -10
\pinlabel $\mathbf{y}_1$ at 203 -10
\endlabellist
\centering
\includegraphics[scale=0.8]{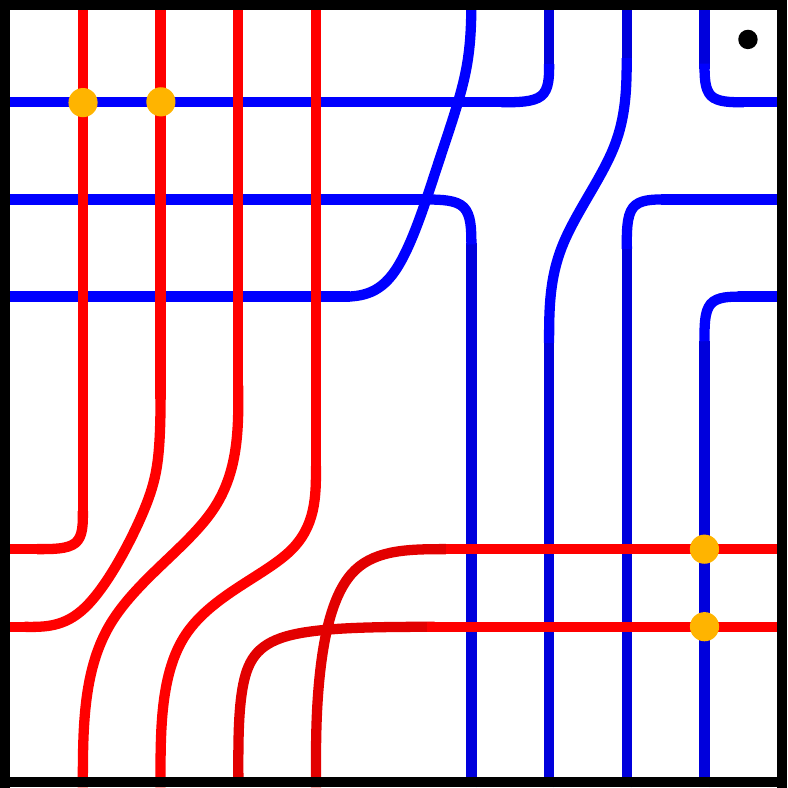}
\vspace{1\intextsep}
\caption{The four generators of $H_{\ast}(\widehat{\textit{CFA}}(M, \mu, \lambda) \boxtimes \widehat{\textit{CFD}}(N, 2\phi_1 + \phi_0, -\phi_1, \s_0))$.}
\label{fig:LargeT23boxN0}
\end{figure}

The gradings take values in $\langle \left(\tfrac{3}{2}; 0, 1 \right) \rangle \backslash G / \langle \left(-1; 4, -2 \right) \rangle$, and are given by the following:
\begin{align*}
\text{gr}(\mathbf{x}_1 \boxtimes \mathbf{a}_1) &= \text{gr}(\mathbf{x}_1)\text{gr}(\mathbf{a}_1) \\
&= \langle \left(\tfrac{3}{2}; 0, 1 \right) \rangle \backslash \left(0; 0, 0 \right) / \langle \left(-1; 4, -2 \right) \rangle \\
\text{gr}(\mathbf{x}_1 \boxtimes \mathbf{a}_2) &= \text{gr}(\mathbf{x}_1)\text{gr}(\mathbf{a}_2) \\
&= \langle \left(\tfrac{3}{2}; 0, 1 \right) \rangle \backslash \left(\tfrac{1}{2}; -2, 1 \right) / \langle \left(-1; 4, -2 \right) \rangle \\
\text{gr}(\mathbf{y_1} \boxtimes \mathbf{b}_1) &= \text{gr}(\mathbf{y_1})\text{gr}(\mathbf{b}_1) \\
&= \langle \left(\tfrac{3}{2}; 0, 1 \right) \rangle \backslash \left(-\tfrac{1}{2}; \tfrac{1}{2}, \tfrac{1}{2} \right) \left(-\tfrac{1}{2}; -\tfrac{3}{2}, \tfrac{1}{2} \right) / \langle \left(-1; 4, -2 \right) \rangle \\
&= \langle \left(\tfrac{3}{2}; 0, 1 \right) \rangle \backslash \left(0; -1, 1 \right) / \langle \left(-1; 4, -2 \right) \rangle \\
\text{gr}(\mathbf{y_1} \boxtimes \mathbf{b}_2) &= \text{gr}(\mathbf{y_1})\text{gr}(\mathbf{b}_2) \\
&= \langle \left(\tfrac{3}{2}; 0, 1 \right) \rangle \backslash \left(-\tfrac{1}{2}; \tfrac{1}{2}, \tfrac{1}{2} \right) \left(-\tfrac{1}{2}; -\tfrac{7}{2}, \tfrac{3}{2} \right) / \langle \left(-1; 4, -2 \right) \rangle \\
&= \langle \left(\tfrac{3}{2}; 0, 1 \right) \rangle \backslash \left(\tfrac{3}{2}; -3, 2 \right) / \langle \left(-1; 4, -2 \right) \rangle
\end{align*}

Acting over $\Q$ to make each $\text{Spin}^{c}$ component equal to $(0,0)$ yields the following:

\begin{align*}
\text{gr}_{\Q}(\mathbf{x}_1 \boxtimes \mathbf{a}_1) &= \langle (\tfrac{3}{2}; 0, 1) \rangle \backslash (\textcolor{violet}{0}; 0, 0) / \langle (-1; 4, -2) \rangle \\
\text{gr}_{\Q}(\mathbf{x}_1 \boxtimes \mathbf{a}_2) &= \langle (\tfrac{3}{2}; 0, 1) \rangle \backslash (\tfrac{1}{2}; -2, 1) / \langle (-1; 4, -2) \rangle \\
&= \langle (\tfrac{3}{2}; 0, 1) \rangle \backslash (\tfrac{1}{2}; -2, 1)(-\tfrac{1}{2}; 2, -1) / \langle (-1; 4, -2) \rangle\\
&= \langle (\tfrac{3}{2}; 0, 1) \rangle \backslash (\textcolor{violet}{0}; 0, 0) / \langle (-1; 4, -2) \rangle\\
\text{gr}_{\Q}(\mathbf{y_1} \boxtimes \mathbf{b}_1) &= \langle (\tfrac{3}{2}; 0, 1) \rangle \backslash (0; -1, 1) / \langle (-1; 4, -2) \rangle \\
&= \langle (\tfrac{3}{2}; 0, 1) \rangle \backslash (-\tfrac{3}{4}; 0, -\tfrac{1}{2})(0; -1, 1)(-\tfrac{1}{4}; 1, -\tfrac{1}{2}) / \langle (-1; 4, -2) \rangle \\
&= \langle (\tfrac{3}{2}; 0, 1) \rangle \backslash (-\tfrac{5}{4}; -1, \tfrac{1}{2})(-\tfrac{1}{4}; 1, -\tfrac{1}{2}) / \langle (-1; 4, -2) \rangle \\
&= \langle (\tfrac{3}{2}; 0, 1) \rangle \backslash (\textcolor{violet}{-\tfrac{3}{2}}; 0, 0) / \langle (-1; 4, -2) \rangle \\
\text{gr}_{\Q}(\mathbf{y_1} \boxtimes \mathbf{b}_2) &= \langle (\tfrac{3}{2}; 0, 1) \rangle \backslash (\tfrac{3}{2}; -3, 2) / \langle (-1; 4, -2) \rangle \\
&= \langle (\tfrac{3}{2}; 0, 1) \rangle \backslash (-\tfrac{3}{4}; 0, -\tfrac{1}{2})(\tfrac{3}{2}; -3, 2)(-\tfrac{3}{4}; 3, -\tfrac{3}{2}) / \langle (-1; 4, -2) \rangle \\
&= \langle (\tfrac{3}{2}; 0, 1) \rangle \backslash (-\tfrac{3}{4}; -3, \tfrac{3}{2})(-\tfrac{3}{4}; 3, -\tfrac{3}{2}) / \langle (-1; 4, -2) \rangle \\
&= \langle (\tfrac{3}{2}; 0, 1) \rangle \backslash (\textcolor{violet}{-\tfrac{3}{2}}; 0, 0) / \langle (-1; 4, -2) \rangle
\end{align*}

While not necessary to generate the desired surgery obstruction, the grading differences for the other four generators may be of independent interest. Performing the same methods for the pairing of $A(\theta_M)$ and $D(\theta_{h(N, \s_1)})$ yields the configuration in Figure \ref{fig:T23N1BoxLift}. The four surviving generators are $\mathbf{x}_1 \boxtimes \mathbf{z}_1$, $\mathbf{y}_1 \boxtimes \mathbf{w}_1$, $\mathbf{x}_1 \boxtimes \mathbf{z}_3$, and $\mathbf{y_1} \boxtimes \mathbf{w}_2$. Knowing these generators, we now compute their refined gradings in $G_{A, \Q}(M, \s) \times_{G_{\Q}} G_{D, \Q}(h(N, \s_1))$ using the information from Subsections \ref{subsec:M} and \ref{subsec:N0}.

The gradings take values in $\langle \left(\tfrac{3}{2}; 0, 1 \right) \rangle \backslash G / \langle \left(-3; 4, -2 \right) \rangle$, and are given by the following:
\begin{align*}
\text{gr}(\mathbf{x}_1 \boxtimes \mathbf{z}_1) &= \text{gr}(\mathbf{x}_1)\text{gr}(\mathbf{z}_1) \\
&= \langle (\tfrac{3}{2}; 0, 1) \rangle \backslash (0; 0, 0) / \langle (-3; 4, -2) \rangle \\
\text{gr}(\mathbf{y}_1 \boxtimes \mathbf{w}_1) &= \text{gr}(\mathbf{y}_1)\text{gr}(\mathbf{w}_1) \\
&= \langle (\tfrac{3}{2}; 0, 1) \rangle \backslash (-\tfrac{1}{2}; \tfrac{1}{2}, \tfrac{1}{2}) (-\tfrac{1}{2}; -\tfrac{5}{2}, \tfrac{1}{2}) / \langle (-3; 4, -2) \rangle \\
&= \langle (\tfrac{3}{2}; 0, 1) \rangle \backslash (\tfrac{1}{2}; -2, 1) / \langle (-3; 4, -2) \rangle \\
\text{gr}(\mathbf{x}_1 \boxtimes \mathbf{z}_3) &= \text{gr}(\mathbf{x}_1)\text{gr}(\mathbf{z}_3) \\
&= \langle (\tfrac{3}{2}; 0, 1) \rangle \backslash (-\tfrac{1}{2}; -1, 0) / \langle (-3; 4, -2) \rangle \\
\text{gr}(\mathbf{y}_1 \boxtimes \mathbf{w}_2) &= \text{gr}(\mathbf{y}_1)\text{gr}(\mathbf{w}_2) \\
&= \langle (\tfrac{3}{2}; 0, 1) \rangle \backslash (-\tfrac{1}{2}; \tfrac{1}{2}, \tfrac{1}{2}) (\tfrac{3}{2}; -\tfrac{7}{2}, \tfrac{3}{2}) / \langle (-3; 4, -2) \rangle \\
&= \langle (\tfrac{3}{2}; 0, 1) \rangle \backslash (\tfrac{7}{2}; -3, 2) / \langle (-3; 4, -2) \rangle
\end{align*}

\begin{figure}
\vspace{1\intextsep}
\labellist
\small\hair 2pt
\pinlabel $\mathbf{x}_1$ at -10 189
\pinlabel $\mathbf{x}_2$ at -10 168
\pinlabel $\mathbf{x}_3$ at -10 147
\pinlabel $\mathbf{z}_1$ at 18 220
\pinlabel $\mathbf{z}_2$ at 35 220
\pinlabel $\mathbf{z}_3$ at 52 220
\pinlabel $\mathbf{z}_4$ at 69 220
\pinlabel $\mathbf{w}_2$ at 185 92
\pinlabel $\mathbf{w}_1$ at 185 74
\pinlabel $\mathbf{y}_4$ at 104 28
\pinlabel $\mathbf{y}_3$ at 121 28
\pinlabel $\mathbf{y}_2$ at 138 28
\pinlabel $\mathbf{y}_1$ at 155 28
\endlabellist
\centering
\includegraphics[scale=0.7]{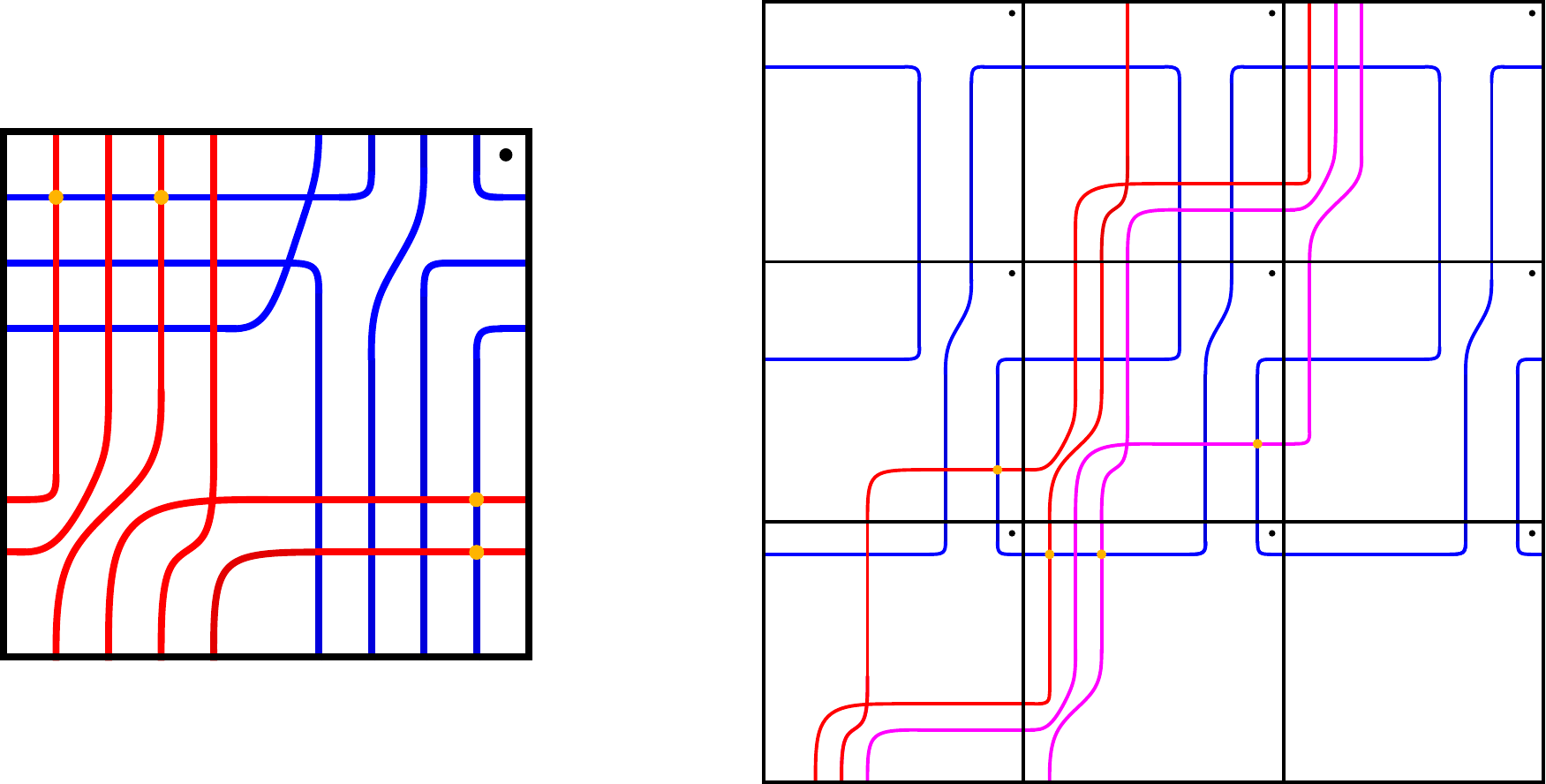}
\vspace{1\intextsep}
\caption{The four generators of $H_{\ast}(\widehat{\textit{CFA}}(M, \mu, \lambda) \boxtimes \widehat{\textit{CFD}}(N, 2\phi_1 + \phi_0, -\phi_1, \s_1))$.}
\label{fig:T23N1BoxLift}
\end{figure}

Acting over $\Q$ to make each $\text{Spin}^{c}$ component equal to $(0,0)$ yields the following:

\begin{align*}
\text{gr}_{\Q}(\mathbf{x}_1 \boxtimes \mathbf{z}_1) &= \langle (\tfrac{3}{2}; 0, 1) \rangle \backslash (\textcolor{violet}{0}; 0, 0) / \langle (-3; 4, -2) \rangle \\
\text{gr}_{\Q}(\mathbf{y}_1 \boxtimes \mathbf{w}_1) &= \langle (\tfrac{3}{2}; 0, 1) \rangle \backslash (\tfrac{1}{2}; -2, 1) / \langle (-3; 4, -2) \rangle \\
&= \langle (\tfrac{3}{2}; 0, 1) \rangle \backslash (\tfrac{1}{2}; -2, 1)(-\tfrac{3}{2}; 2, -1) / \langle (-3; 4, -2) \rangle\\
&= \langle (\tfrac{3}{2}; 0, 1) \rangle \backslash (\textcolor{violet}{-1}; 0, 0) / \langle (-3; 4, -2) \rangle\\
\text{gr}_{\Q}(\mathbf{x}_1 \boxtimes \mathbf{z}_3) &= \langle (\tfrac{3}{2}; 0, 1) \rangle \backslash (-\tfrac{1}{2}; -1, 0) / \langle (-3; 4, -2) \rangle \\
&= \langle (\tfrac{3}{2}; 0, 1) \rangle \backslash (\tfrac{3}{4}; 0, \tfrac{1}{2})(-\tfrac{1}{2}; -1, 0)(-\tfrac{3}{4}; 1, -\tfrac{1}{2}) / \langle (-3; 4, -2) \rangle \\
&= \langle (\tfrac{3}{2}; 0, 1) \rangle \backslash (\tfrac{3}{4}; -1, \tfrac{1}{2})(-\tfrac{3}{4}; 1, -\tfrac{1}{2}) / \langle (-3; 4, -2) \rangle \\
&= \langle (\tfrac{3}{2}; 0, 1) \rangle \backslash (\textcolor{violet}{0}; 0, 0) / \langle (-3; 4, -2) \rangle
\end{align*}
\begin{align*}
\text{gr}_{\Q}(\mathbf{y}_1 \boxtimes \mathbf{w}_2) &= \langle (\tfrac{3}{2}; 0, 1) \rangle \backslash (\tfrac{7}{2}; -3, 2) / \langle (-3; 4, -2) \rangle \\
&= \langle (\tfrac{3}{2}; 0, 1) \rangle \backslash (-\tfrac{3}{4}; 0, -\tfrac{1}{2})(\tfrac{7}{2}; -3, 2)(-\tfrac{9}{4}; 3, -\tfrac{3}{2}) / \langle (-3; 4, -2) \rangle \\
&= \langle (\tfrac{3}{2}; 0, 1) \rangle \backslash (\tfrac{5}{4}; -3, \tfrac{3}{2})(-\tfrac{9}{4}; 3, -\tfrac{3}{2}) / \langle (-3; 4, -2) \rangle \\
&= \langle (\tfrac{3}{2}; 0, 1) \rangle \backslash (\textcolor{violet}{-1}; 0, 0) / \langle (-3; 4, -2) \rangle
\end{align*}
\bigskip

With the desired grading differences at hand, we now establish the third lemma.

\begin{lemma}
Let $X = (S^3 \setminus \nu T(2,3)) \cup_h N$, where $h$ is any slope 2 cyclic gluing. Then there exist generators $\mathbf{x}, \mathbf{y}$ of $\widehat{\textit{HF}}(X)$ such that $\s(\mathbf{x}), \s(\mathbf{y}) \in \pi^{-1}(\s \times \s_0)$ and $\text{gr}_{\Q}(\mathbf{x})-\text{gr}_{\Q}(\mathbf{y}) = \tfrac{3}{2}$.
\label{lem:MLemgr}
\end{lemma}

\begin{proof}
Recall that any slope 2 gluing satisfying the cyclic condition of Proposition \ref{prop:pairhomology} induces
\[
[(h^n)_{\ast}] = \begin{pmatrix}
1 & n \\
2 & 2n-1 \\
\end{pmatrix}.
\]
on homology for some $n \in \Z$. We saw from Subsection \ref{subsec:dehntwist} that it is obtained from the base slope 2 gluing $h$ by the pre-composition
\[
[(h^n)_{\ast}] = [h_{\ast} \circ T^n] = \begin{pmatrix}
1 & 0 \\
2 & -1 \\
\end{pmatrix}
\begin{pmatrix}
1 & n \\
0 & 1 \\
\end{pmatrix}.
\]

For $X^n = M \cup_{h^n} N$, the pairing theorem implies \\
$\widehat{\textit{HF}}(X^n, \mathfrak{t}) \cong H_{\ast}(\widehat{\textit{CFA}}(M, \mu, \lambda) \boxtimes \widehat{\textit{CFD}}(N, (h^n)^{-1}_{\ast}(\lambda), (h^n)^{-1}_{\ast}(\mu), \s_0))$ for $\mathfrak{t} \in \pi^{-1}(\s \times \s_0)$. The required type D structure for pairing is $\widehat{\textit{CFD}}(N, 2\phi_1 + (1-2n)\phi_0, -\phi_1 + n\phi_0, \s_0)$, whose type A realization can be obtained from that of $\widehat{\textit{CFD}}(N, \phi_1 - n\phi_0, \phi_0, \s_0)$ by using precisely the same series of Dehn twists and reflections from Subsection \ref{subsec:N0}. Since $\widehat{\textit{CFD}}(N, \phi_1 - n\phi_0, \phi_0, \s_0) \cong \widehat{\textit{CFD}}(N, \phi_1, \phi_0, \s_0)$ because $N$ is a Heegaard Floer homology solid torus, we have that the type A realizations of $\widehat{\textit{CFD}}(N, 2\phi_1 + (1-2n)\phi_0, -\phi_1 + n\phi_0, \s_0)$ and $\widehat{\textit{CFD}}(N, 2\phi_1 + \phi_0, -\phi_1, \s_0)$ agree. Further, the decorated graph representations and refined gradings from Subsection \ref{subsec:N0}, as well as the relative $\Q$-grading differences (between generators belonging to the same pre-image $\pi^{-1}(\s \times \s_i)$) from Subsection \ref{subsec:grcomputation} all hold for each gluing $h^n$, regardless of $n \in \Z$. 

These grading differences were computed using the prototype pairing theorem for train tracks, whose use is currently unjustified. They will correspond to grading differences of $\widehat{\textit{HF}}(X)$ if the bordered invariants in pairing satisfy the mild hypotheses of \cite[Theorem 16]{HRW16}. Specifically, we require that $A(\theta_M)$ is reduced and that $D(\theta_{h(N, \s_0)})$ is special bounded. The realization $A(\theta_M)$ is reduced since it is generated from a decorated graph for which no edges are labeled with $\varnothing$. A train track is special bounded if it is bounded and almost reduced, meaning that its underlying decorated graph does not contain an oriented cycle and any edges labeled with $\varnothing$ occur in specific configurations. The underlying decorated graph for $D(\theta_{h(N, \s_0)})$ does not have any edges labeled with $\varnothing$, and also does not contain an oriented cycle (even though $D(\theta_{N, \s_0})$ does).

With equal $\text{spin}^{c}$ components, we can recover the $\Q$-grading difference between $\mathbf{x}_1 \boxtimes \mathbf{a}_1$ and $\mathbf{y}_1 \boxtimes \mathbf{b}_1$ as the difference between their Maslov components: $\text{gr}_{\Q}(\mathbf{x}_1 \boxtimes \mathbf{a}_1)-\text{gr}_{\Q}(\mathbf{y}_1 \boxtimes \mathbf{b}_1) = \tfrac{3}{2}$. Together with the Dehn twisting invariance of $\widehat{\textit{CFD}}(N, \phi_1, \phi_0, \s_0)$, this establishes the lemma.
\end{proof}

\begin{remark}
We caution the reader that the techniques used to prove this lemma do \textbf{not} show that the Dehn twisting invariance of $\widehat{\textit{CFD}}(N, \phi_1, \phi_0)$ in pairing provides an integral family of manifolds with relatively-graded, isomorphic $\widehat{\textit{HF}}$. All that we show is that the relative $\Q$-grading differences between generators $\mathbf{x}, \mathbf{y}$ with $\s(\mathbf{x}), \s(\mathbf{y}) \in \pi^{-1}(\s \times \s_i)$ are independent of Dehn twisting $N$ along $\phi_0$. We should expect to see the relative $\Q$-grading differences grow between generators providing Floer homology supported in $\text{spin}^{c}$ structures from different pre-images.
\end{remark}
\smallskip

\subsection{Proof of the main theorem}
\label{subsec:mainproof}

We are now equipped to prove Theorem \ref{thm:main}. \\

\noindent\textit{Proof.}
Suppose $X = (S^3 \setminus \nu J) \cup_h N$ is realized as $8$-surgery along $K$ with $g(K)=2$. If $J$ is the unknot, then Lemma \ref{lem:MLemFinite} using Doig's classification together with \cite[Theorem 1.6]{NZ18} implies that $X = S^3_8(T(2,5))$. This manifold is an L-space, and is the Seifert fibered manifold $(-1; \frac{1}{2}, \frac{1}{2}, \frac{2}{5})$ with base orbifold $S^2$.

Suppose for the sake of contradiction that some non-trivial $J \subset S^3$ gives rise to $X$. Then Lemma \ref{lem:MLemS3} implies $J=T(2,3)$ and that $X$ is an L-space. As an L-space, the relative $\Q$-grading differences for generators of $\widehat{\textit{HF}}(X)$ are given by differences of the $d$-invariants. Since $K$ is a genus two L-space knot, we have $\widehat{\textit{HFK}}(K) = \widehat{\textit{HFK}}(T(2,5))$ and so the $d$-invariants of surgery are 
\begin{align*}
d(S^3_8(K),[s]) = \left\{
	\begin{array}{l c}
	 -1/8 & s \equiv 5 \,\, (\text{mod} \,\, 8) \\
	 1/4 & s \equiv 6 \,\, (\text{mod} \,\, 8) \\
	 -9/8 & s \equiv 7 \,\, (\text{mod} \,\, 8) \\
	 -1/4 & s \equiv 0 \,\, (\text{mod} \,\, 8) \\
	 -9/8 & s \equiv 1 \,\, (\text{mod} \,\, 8) \\
	 1/4 & s \equiv 2 \,\, (\text{mod} \,\, 8) \\
	 -1/8 & s \equiv 3 \,\, (\text{mod} \,\, 8) \\
	 -1/4 & s \equiv 4 \,\, (\text{mod} \,\, 8)
	\end{array}
	\right.
\end{align*}
Lemma \ref{lem:MLemgr} shows that there must be generators $\mathbf{x}, \mathbf{y}$ for $\widehat{\textit{HF}}(X)$ such that $\text{gr}_{\Q}(\mathbf{x})-\text{gr}_{\Q}(\mathbf{y}) = \tfrac{3}{2}$. This is impossible given $d(S^3_8(K), [s])$ above, which is the contradiction we sought.
\qed

\section{The situation for $Y \neq S^3$}
\label{sec:YnotS3}

As alluded to in the introduction, the immersed curves techniques were stated for knots $J \subset S^3$, but hold more generally for knots $J$ in integer homology sphere L-spaces. When $Y$ is an integer homology sphere L-space different from $S^3$, we can use Corollary \ref{cor:irreduciblecomplement} together with a corollary of Baldwin and Vela-Vick. If $J \subset Y$ is a nullhomologous knot with irreducible complement and dim $\widehat{\textit{HFK}}(Y, K) = 3$, their work implies that $Y \setminus \nu J \cong S^3 \setminus \nu T(2, 3)$ \cite{BV18}. Thus, the proof of Theorem \ref{thm:main} still holds. 

We used Floer homology considerations to constrain the form of $\widehat{\textit{HF}}(Y \setminus \nu J)$, which allowed us to determine the possible knots $J$ together with known properties coming from the algorithm from \cite[Proposition 47]{HRW18}. This algorithm is the immersed curves form of a theorem of Lipshitz, Ozsv\'ath, and Thurston that computes the type D structure for a knot complement from $CFK^-(Y, K)$, which holds for integer homology sphere L-spaces $Y$ \cite{LOT18b}. In the absence of a full analogue of this algorithm, we could potentially handle the remaining cases for $Y$ just by knowing what form the essential curve components must take in $\overline{T}_M$. In particular, in \cite[Section 5]{KWZ20} it is suggested that the essential curve component potentially picks up a non-trivial local system if the map taking vertical homology to horizontal homology in the full knot Floer complex is interesting.

Still, we can narrow down some properties of $Y$ through homological means. Consider the pairing $X=M \cup_h N$, where $M$ is a rational homology solid torus and $h$ is a slope $p$ gluing, defined as in Definition \ref{def:X}. Let $\lambda_M$ denote the rational longitude of $M$, and $\mu_M$ the respective dual curve. Then we have 
\[
|H_1(X)| = 4d|H|\Delta(\lambda_M, h_{\ast}(\phi_0)),
\]
where $d = o(\lambda_M)$ in $H_1(M)$ and $H$ is the torsion subgroup of $H_1(M)$ \cite[Section 3]{BGW13}. Since $h$ has gluing slope $p$, we have $\Delta(\lambda_M, h_{\ast}(\phi_0)) = p$, and so $|H_1(X)| = 4pd|H| = 8$.

Recall that $\lambda_M$ generates Ker$(H_1(\partial M) \rightarrow H_1(M)) \cong d\Z \oplus 0 \subset \Z \oplus \Z$.  The rational longitude includes as $i_{\ast}(\lambda_M) \in H \subset H_1(M)$ with finite order (and is unique among slopes in $H_1(\partial M)$ with this property \cite{Wat12}). Then we must have $d = 1$, as otherwise $d > 1$ implies $|H| \geq d > 1$. Thus, $\lambda_M$ includes as a null-homologous curve.

Suppose $|H| = 2$, so that $p = 1$. A similar computation to that in the proof of Proposition \ref{prop:pairhomology} shows that the induced map $h_{\ast}$ in the Mayer-Vietoris sequence for $X = M \bigcup_h N$ does not interact with $H$. This implies that $H_1(X)$ carries $H$ as a free summand. However $|H| = 2$ is not relatively prime to the orders of the other summands of $H_1(X)$, and so we must have $|H|=1$ to have a pairing with cyclic $H_1(X)$. We are left with $|H| = 1$ and $p = 2$, which has $M$ as the exterior of a knot in an integer homology sphere.

\bibliographystyle{alpha}
\bibliography{bibliography}

\end{document}